\documentclass[10pt,a4paper,reqno]{amsart}
\usepackage[latin1]{inputenc}
\usepackage{amsmath}
\usepackage{amssymb} 
\usepackage{amsfonts, verbatim,hyperref, multirow,diagbox,tikz}
\usepackage[all,cmtip]{xy}
\usepackage{bbold,bbm}

\makeatletter
\def\subsection{\@startsection{subsection}{3}%
	\z@{.9\linespacing\@plus.7\linespacing}{.1\linespacing}%
	{\normalfont\bfseries}}
\makeatother

\hypersetup{
	colorlinks,
	citecolor=blue,
	filecolor=black,
	linkcolor=blue,
	urlcolor=blue,
}

\newtheorem{theorem}{Theorem}
\newtheorem{corollary}{Corollary}
\newtheorem{lemma}{Lemma}

\newtheorem{prop}{Proposition}

\newtheorem{defn}{Definition}

\newcommand{\C}{\mathbb C}
\newcommand{\ztwo}{\mathbb Z/2\mathbb Z}
\newcommand{\nc}{\newcommand}
\nc{\dmo}{\DeclareMathOperator}
\nc{\thm}{\theorem}
\nc{\cor}{\corollary}
\nc{\mc}{\mathcal}
\nc{\mb}{\mathbb}
\nc{\mf}{\mathfrak}
\nc{\ul}{\underline}
\nc{\ol}{\overline}
\nc{\N}{\mb N}
\nc{\R}{\mb R}
\nc{\Z}{\mb Z}
\nc{\Q}{\mb Q}
\nc{\F}{\mathbb{F}}
\nc{\gm}{\gamma}
\nc{\Hom}{Hom}
\nc{\SO}{SO}
\nc{\SL}{SL}
\nc{\Spin}{Spin}
\nc{\sgn}{sgn}
\nc{\half}{\frac{1}{2}}
\nc{\mat}[4]{
	\begin{pmatrix}
		#1 & #2 \\
		#3 & #4
	\end{pmatrix}
}
\nc{\lip}{\langle}
\nc{\rip}{\rangle}
\dmo{\ind}{ind}
\dmo{\St}{St}
\dmo{\st}{st}
\dmo{\res}{res}
\dmo{\Sq}{Sq}

\nc{\beq}{\begin{equation*}}
\nc{\eeq}{\end{equation*}}

\dmo{\Ind}{Ind}
\dmo{\Mod}{mod}
\dmo{\ver}{ver}
\dmo{\GL}{GL}
\dmo{\Ord}{Ord}
\title{Stiefel Whitney Classes for Real representations of $\GL_2(\F_q)$}
\author{Jyotirmoy Ganguly, Rohit Joshi}
\makeatletter\@addtoreset{chapter}{part}\makeatother%

\address{The Institute of Mathematical Sciences, IV Cross Road, CIT Campus, Taramani, Chennai-600113, Tamil Nadu, India} \email{jyotirmoy.math@gmail.com}  

\address{Bhaskaracharya Pratishthana, 56/14, Erandavane, Damle Path, Off Law College Road,Pune-411004, Maharashtra, India} \email{rohitsj@students.iiserpune.ac.in}  

\begin{document}
	
\maketitle
\begin{abstract}
We compute the total Stiefel Whitney class for a real representation $\pi$ of $\GL_2(\F_q)$, where $q$ is odd.
The obstruction class of  $\pi$ is defined to be the Stiefel Whitney class of lowest
positive degree that does not vanish. We provide an expression for the obstruction class of $\pi$ in terms of its character values if $\det\pi=1$. 
\end{abstract}

\tableofcontents

\section{Introduction}\label{intro}

Let $\pi$ be a real representation of a finite group $G$. One can associate 
%a real vector bundle $\xi(\pi)$  over $BG$ (classifying space of $G$) to $\pi$ and consider 
Stiefel Whitney classes $w_i(\pi)\in H^i(G, \Z/2\Z)$ (see \cite[Section $2.6$, page no. $50$]{benson}) to it such that $w_i(\pi)=0$ if $i>\deg\pi$. We define the ``obstruction class" to be $w_i(\pi)$ where the degree $i>0$, called the obstruction degree of $\pi$, is minimal with $w_i(\pi)\neq 0$. In fact the obstruction degree of $\pi$ turns out to be a power of $2$ (see \cite[Problem $8$-B]{mch}). 
%For a real vector bundle $\xi$ of dimension $n$ if $w_i(\xi) \neq 0$ then there cannot exist $n - i + 1$ linearly independent sections of $\xi$ (see \cite[Section $3.3$]{hatchervbk}).  
In this paper we calculate the total Stiefel Whitney class of a real representaton $\pi$ of $\GL_2(\F_q)$, where $q$ is odd. Moreover, if $w_1(\pi)=0$ we compute the obstruction class of $\pi$ in terms of its character values.

Let $G=\GL_2(\F_q)$ be the general linear group for a vector space of dimension 2 over the field with $q$ elements where $q$ is odd. Let $\pi$ be a real representation of $G$. 
%We use the fact that $$H^*(C_n \times C_n,\ztwo) = \dfrac{ \F_2[s_1, s_2, t_1,t_2]}{(s_1^2,s_2^2)},  $$ and $i^*:H^*(G,\ztwo) \to H^*(D,\ztwo)$ is injective, where $i: D\to G$ denotes the inclusion map. Thus we find $i^*(w_i)$ in place of $w_i$ of $G$.
%
%Given a real representation $\pi$ of $C_n\times C_n,\, n\equiv 0\pmod 4$ and $(x,y)\in \{0,1\}^2\setminus \{0,0\}$, we define
%$$m_{xy}=\frac{1}{8}\left(\chi_{\varphi}(0,0)+(-1)^x\chi_{\varphi}(n/2,0)+(-1)^y\chi_{\varphi}(0,n/2)+(-1)^{x+y}\chi_{\varphi}(n/2,n/2) \right).$$
%\begin{enumerate}
%	%	\item 
%	%	$m'+2m_{00}=\frac{1}{4}(\chi_{\pi}(0,0)+\chi_{\pi}(n/2,0)+\chi_{\pi}(0,n/2)+\chi_{\pi}(n/2,n/2))$
%	\item 
%	$m_{10}=\frac{1}{8}(\chi_{\pi}(0,0)-\chi_{\pi}(n/2,0)+\chi_{\pi}(0,n/2)-\chi_{\pi}(n/2,n/2))$,
%	\item 
%	$m_{01}=\frac{1}{8}(\chi_{\pi}(0,0)+\chi_{\pi}(n/2,0)-\chi_{\pi}(0,n/2)-\chi_{\pi}(n/2,n/2))$,
%	\item 
%	$m_{11}=\frac{1}{8}(\chi_{\pi}(0,0)-\chi_{\pi}(n/2,0)-\chi_{\pi}(0,n/2)+\chi_{\pi}(n/2,n/2))$.
%	%		\item 
%	%		$m_1m_2+m_2m_3+m_3m_1=(\delta_{10}\cdot \delta_{01}+\deg \pi+\deg \pi^{C_n\times C_n}+\delta_{10}+\delta_{01})$.
%\end{enumerate}
%Suppose $w_{2^i}(\pi) = 0$ for $i=0,1, \ldots ,k-1$ when $k>1$, then  
	%	\begin{align*}
	%	w_{2^k}(\pi)=& \binom{m_{10} + m_{11}}{2^{{k-1}}}t_1^{2^{k-1}}+ \binom{m_{01} + m_{11}}{2^{{k-1}}}t_2^{2^{k-1}}\\
	%	&+\left(\binom{m_{10} + m_{11}}{2^{{k-1}}} + \binom{m_{01} + m_{11}}{2^{{k-1}}}+\binom{m_{10} + m_{01}}{2^{{k-1}}}\right) t_1^{2^{k-2}}t_2^{2^{k-2}},
	%	\end{align*}
	%	where the terms $m_{01}, m_{10}, m_{11}$ are as in Lemma \ref{character.formula}. 	
Take $z_x=\mat x00x$ and $t_{(x,y)}=\mat x00y$. Consider the terms

\begin{enumerate}
	%	\item 
	%	$m'+2m_{00}=\frac{1}{4}(\chi_{\pi}(0,0)+\chi_{\pi}(n/2,0)+\chi_{\pi}(0,n/2)+\chi_{\pi}(n/2,n/2))$
	%	
	%\item 
	%	$m_{10}=\frac{1}{8}(\chi_{\pi}(z_1)-\chi_{\pi}(z_{-1}))$
	\item 
	$m_{10}=\frac{1}{8}(\chi_{\pi}(z_1)-\chi_{\pi}(z_{-1}))$
	\item 
	$m_{11}=\frac{1}{8}(\chi_{\pi}(z_1)-2\chi_{\pi}(t_{(1,-1)})+\chi_{\pi}(z_{-1}))$.
	%	\item 
	%	$m_1m_2+m_2m_3+m_3m_1=(\delta_{10}\cdot \delta_{01}+\deg \pi+\deg \pi^{C_n\times C_n}+\delta_{10}+\delta_{01})$.
\end{enumerate}
It turns out that $m_{10}, m_{11}\in \Z_{\geq 0}$.

From \cite[Page 30,Theorem 2.23]{tota} we have for $q\equiv 1\pmod 4$ 
$$H^*(\GL_2(\F_q),\Z/2\Z) = \ztwo[e_1,e_2,c_1,c_2]/(e_1^2=0,e_2^2=0), $$
where $e_1,e_2,c_1,c_2$ are elements of degree $1,3,2,4$ respectively.

%For arriving at the following theorem, first we find the $2^k$-th Stiefel Whitney class for representations of $C_{q-1} \times C_{q-1}$ which is the diagonal subgroup $D$ in $G$, provided $w_{2^i}(\pi) = 0$ for $0 \leq i \leq k-1$. 
A real representation $\pi$ of $G$ is called achiral if $\det(\pi)=1$. Otherwise it is called chiral.
We define 
\[
\delta=
\begin{cases}
0,\quad \text{if $\pi$ is achiral},\\
1,\quad \text{if $\pi$ is chiral}.
\end{cases}
\]
For $a\in \mathbb{Z}_{> 0}$, write $\upsilon_2(a)$ to denote the highest power of $2$ dividing $a$.  Now we state our main results.
\begin{theorem}\label{tsw}
	Let $\pi$ be a real representation of $\GL_2(\F_q)$ such that $q\equiv 1\pmod 4$. Then 
	$$w(\varphi)=(1+\delta e_1)(1+c_1+c_2)^{m_{10}}(1+c_1)^{m_{11}}.$$
\end{theorem}

\begin{theorem} \label{main}  
Let $\pi$ be a real achiral representation of $\GL_2(\F_q)$, with $q \equiv 1 \pmod 4$ and $k=\min{\{\upsilon_2(m_{10}+m_{11}),\upsilon_2(m_{01})+1\}}$. Then the obstruction class of $\rho$ is
\[
w_{2^{k+1}}(\pi) =
\begin{cases}
c_1,\quad \text{if $k=0$},\\\\
\left(\dfrac{m_{10} + m_{11}}{2^k}\right)c_1^{2^k}+\left(\dfrac{ m_{10}}{2^{k-1}}\right)c_2^{2^{k-1}},\quad \text{if $k\geq 1$}.
\end{cases}
\]

%	\begin{enumerate}
%		\item $w_i(\rho) = 0 $, for $1 \leq i \leq 2^{k+1}-1$,
%		\item The obstruction class is $$w_{2^{k+1}}(\rho) = \left(\dfrac{m_{01} + m_{11}}{2^k}\right)\cdot(t_1^{2^k} + t_2^{2^k})+\left(\dfrac{ m_{01}}{2^{k-1}}\right)\cdot (t_1^{2^{k-1}}t_2^{2^{k-1}}).$$
%	\end{enumerate}

\end{theorem}

We also have a similar results for a real representation $\pi$ of $\GL_2(\F_q)$ where $q \equiv 3 \pmod 4$ (See Theorem \ref{tsw2} and Theorem \ref{main2} ). We obtain explicit expressions for $w_1(\pi)$ and $w_2(\pi)$ where $\pi$ varies over all orthogonally irreducible  representations (see Section \ref{oir})  of $G$ (see Table \ref{table1}). Moreover, we calculate their obstruction classes provided $w_1(\pi)=w_2(\pi)=0$. (see Table \ref{table2}). 

Let $C_n$ denote the (additive) cyclic group of order $n$.
The diagonal subgroup $D$ of $G$ is isomorphic to $C_{q-1} \times C_{q-1}$. Moreover, the restriction map $\res: H^*(G, \ztwo)\to H^*(D,\ztwo)$ is injective (see Section \ref{detect}). In short we say
the diagonal subgroup $D$ detects the mod $2$ cohomology of $G$.
Therefore our problem reduces to finding total Stiefel Whitney class and Obstruction class for the representation $\pi\mid_D$. 

As an application we show that an anisotropic torus $M$ does not detect the mod $2$ cohomolgy of $G$ (see Proposition \ref{app1}). If the obstruction degree of $\pi$ is $2^l$ then we compute the Stiefel Whitney classes $w_j(\pi)$, where $2^l<j<2^{l+1}$, using Wu's formula (see Theorems \ref{app2} and \ref{app3}). 

%We give a conjecture for the first non-zero Stiefel Whitney class of a representation of the group $C_2\times C_2\times C_2$ (see Conjecture \ref{con}).

%Let $r_G$ denote the regular representation of $G$.
%Following \cite{kahn} we define
%\begin{equation}
%\nu(G)=\text{min}\{n>0\mid w_{2^{n-1}}(r_G)\neq 0\}.
%\end{equation}
%We prove that $\nu(G)=\upsilon_2(|G|)$.
%In particular, for $q\equiv 1\pmod 4$, the obstruction class is
%\begin{equation*}
%w_{2^{m-1}}(r_G)=t_1^{2^{m-2}}+t_2^{2^{m-2}}+t_1^{2^{m-3}}t_2^{2^{m-3}},
%\end{equation*}
%where $m=\upsilon_2(|G|)$. One obtains a similar result for $q\equiv 3\pmod 4$ (see Theorem \ref{reg}).

%\begin{cases}
%\quad \text{if $q\equiv 1\pmod 4$},\\
%v_1^{2^{m-1}}+v_2^{2^{m-1}}+v_1^{2^{m-2}}v_2^{2^{m-2}},\quad \text{if $q\equiv 3\pmod 4$},
%\end{cases}
%\]

This paper is arranged as follows. In Section \ref{2}, after establishing notations, we provide some results on binomial coefficients. We review the theory of Stiefel Whitney classes for representations of finite groups in Section \ref{swc.section}. 
%We compute Stiefel Whitney classes for real representations of cyclic groups in Section \ref{3}.  
In Section \ref{sectsw} we compute the total Stiefel Whitney class of real representations of bicyclic groups and $\GL_2(\F_q)$ (see Lemma \ref{lem0} and Theorem \ref{q3tot}). Section \ref{bicyclic} is dedicated to computations of the obstruction classes for real representations of bicyclic groups in terms of character values (see for instance Theorems \ref{obstruction} and \ref{c2obs}). We apply these to obtain the required results for $G$.
 Section \ref{5}  contains a proof of Theorem \ref{main}. For an orthogonally irreducible representation $\pi$ we calculate $w_1(\pi), w_2(\pi)$ and summerize the results in Table \ref{table1}. We also provide a table ( Table \ref{table2}) showing obstruction classes for certain representaions of $\GL_2(\F_q)$. Finally in Section \ref{app} we provide some applications (see Theorems \ref{app2}, \ref{app3} and \ref{app1}).  
%We obtain the obstruction class for regular representation of $\GL_2(\F_q)$ in Section \ref{6}. 

{\bf Acknowledgements:} The authors would like to thank Steven Spallone for helpful conversations. The first author of this paper was supported by a post doctoral fellowship from IMSc, Chennai. The second author of this paper was supported by a post doctoral fellowship from Bhaskaracharya Pratishthan, Pune and a fellowship from ARSI (the Foundation for the Advancement of Arts and Sciences from India, Inc.), an organization founded by Ravi Kulkarni.

\section{Notation and Preliminaries} \label{2}

\subsection{Orthogonal Representations and Spinoriality}\label{oir}

A complex representation $(\pi,V)$ of a finite group $G$ is called orthogonal if it preserves a non-degenerate symmetric bilinear form. An orthogonal representation $\pi$ is called spinorial if it can be lifted to $\mathrm{Pin}(V)$, the topological double cover of $\mathrm{O}(V)$.  
See \cite{jyoti} for reference.

Here we present a brief review on real and complex representations of a finite group $G$. For details and proofs we refer the reader to \cite[Section II.6]{BrokerDieck}.
%We now review the relationship between real and complex representations of $G$; see   \cite[Section II.6]{BrokerDieck} for proofs. Here ``real" and ``complex" refer to the field of scalars of the underlying vector space of the representation.
A real (or complex) representation means the underlying vector space is real (or complex).
For a complex representation $(\pi,V)$ of $G$ we write $(\pi_{\R},V_{\R})$ for the realization of $\pi$. This simply means that we forget the complex structure on $V$ and regard it as a real representation. When $\pi$ is orthogonal, there is a unique real representation $(\pi_0,V_0)$, up to isomorphism, so that $\pi \cong \pi_0 \otimes_{\R} \C$. Observe that a real representation $\pi$ of a finite group is equivalent to an orthogonal representation. 
%For example one can take the invariant non-degenerate symmetric bilinear form to be $$B(v,w)= \sum\limits_{g \in G} \lip g\cdot v, g\cdot w \rip,$$
%where $\lip, \rip$ is the standard inner product on $\R^n$.
%We have 
%$\pi_{\mathbb{R}}=(\pi\oplus \pi^{\vee})_0$ as real representations.
%Let $(\pi,V)$ be a complex representation of $G$.   Write $(\pi_{\R},V_{\R})$ for the realization of $\pi$, meaning that we simply forget the complex structure on $V$ and regard it as a real representation.  When $\pi$ is orthogonal, there is a unique real representation $(\pi_0,V_0)$, up to isomorphism, so that $\pi \cong \pi_0 \otimes_{\R} \C$.

%\begin{thm} \label{dichotomy} If $\rho$ is an irreducible real representation of $G$, then exactly one of the following must be true:
%	\begin{enumerate}
%		\item $\rho \cong \pi_0$, for some irreducible orthogonal   complex representation $\pi$ of $G$.
%		\item $\rho \cong \pi_\R$, for some irreducible complex representation $\pi$ of $G$ which is not orthogonal.
%	\end{enumerate}
%\end{thm}

%Given a complex representation $(\pi,V)$, write $S(\pi)$ for the representation $\pi \oplus \pi^\vee$ on $V \oplus V^\vee$.
%Then $S(\pi)$ is orthogonal, as it preserves the quadratic form  
%\beq
%\mc Q((v,v^*))=\lip v^*,v \rip.
%\eeq
%Moreover, $S(\pi)_0 \cong \pi_\R$. 

Any orthogonal complex representation $\Pi$ of $G$ can be decomposed as
\beq
\Pi = S(\pi) \oplus \bigoplus_j \varphi_j,
\eeq
where $S(\pi)=\pi \oplus \pi^\vee$ and each $\varphi_j$ is irreducible orthogonal and $\pi$ is arbitrary. For details see \cite[Section $2.1$]{spjoshi}.

We say a complex representation $\pi$ is \emph{orthogonally irreducible}, provided $\pi$ is orthogonal, and $\pi$ does not decompose into a direct sum of \emph{orthogonal} representations. Thus, an orthogonal representation $\pi$ is orthogonally irreducible iff $\pi$ is irreducible, or of the form $S(\phi)$ where $\phi$ is irreducible but not orthogonal. We write `OIR' for ``orthogonally irreducible representation".

%Let $(\pi_1,V_1)$, $(\pi_2,V_2)$ be representations of finite groups $G_1,G_2$, respectively. We write $(\pi_1 \boxtimes \pi_2,V_1 \otimes V_2)$ for the external tensor product representation of $G_1 \times G_2$.

For an even integer $n$, write $C_n$ and $\mu_n$ for the additive cyclic  group and multiplicative cyclic group of order $n$ respectively. Let $\zeta_n=e^{\frac{2 \pi i}{n}}$ be the primitive $n^{th}$ root of unity. Let $\chi^j: C_n \to \mu_n$ denote the representation where $\chi^j(1)=\zeta_n^j$ and $\sgn=\chi^{n/2}$. 
%For a complex orthogonal representation $\pi$ of a finite group $G$, there exists a unique real representation $\pi_0$ such that $\pi\cong \pi_0\otimes \C$. Let the representation $\pi_{\R}$ denote the realization of a representation  $\pi$. 
%Let $\chi:C_n\to \C^{\times}$ be a character such that $\chi(1)=\zeta_n^m$. 
We write $\epsilon_{\chi^j}$ to denote the parity of $j$. We say $\chi^j$ is odd (resp. even) if $\epsilon_{\chi^j}$ is odd (resp. even).

%For a representation $\pi$ we define $S(\pi)=\pi\oplus \pi^{\vee}$.
%We say a representation $\pi$ is orthogonally irreducible if $\pi$ is orthogonal and $\pi$ cannot be decomposed into orthogonal subrepresentations. In other words $\pi$ is irreducible orthogonal or it is of the form $S(\phi)$, where $\phi$ is irreducible but not orthogonal.

\subsection{Some Results on Binomial Coefficients}

For non-negative integers $m, n$ and a prime $p$ one gets base $p$ expansions for $m$ and $n$ as $m=\sum\limits_{i=0}^{k}m_ip^i$ and $n=\sum\limits_{i=0}^{k}n_ip^i$. Here we state Lucas theorem (\cite{fine}) which we use extensively in Section $4$.

\begin{theorem}[Lucas Theorem]\label{lucas}
For non-negative integers $m, n$ and a prime $p$ we have
\begin{equation*}
\binom{m}{n}\equiv \prod_{i=0}^{k}\binom{m_i}{n_i}\pmod p.
\end{equation*}
\end{theorem}

\begin{prop}\label{binom1}
	If $\upsilon_2(n)\geq k$, then $\dbinom{n}{2^k}\equiv \dfrac{n}{2^k}\pmod 2$.
\end{prop}

\begin{proof}
	The proof follows from Theorem \ref{lucas}.
\end{proof}

The Vandermonde's identity for binomial coefficients states that
\begin{equation}\label{van}
\binom{n_1+\cdots+n_l}{m}=\sum_{k_1+\cdots+k_l=m}\binom{n_1}{k_1}\binom{n_2}{k_2}\cdots\binom{n_l}{k_l}.
\end{equation} 

%For $n\in \mathbb{Z}^+$, let $\nu(n)$ denote the number of $1$'s apperaing in the binary expansion of $n$. 
We write $\upsilon_2(n)$ to denote the $2$-adic valuation of the $n$. 
%We have 
%\begin{equation}\label{2adic}
%v_2\binom{n}{r}=\#\{\text{carries while adding $r$ and $n-r$ in binary} \}.
%\end{equation}

\begin{prop}
	For $n\in \Z_{\geq 0}$, if $\upsilon_2(n)\geq k$ and $a<2^{k}$, then $\binom{n}{a}\equiv 0\pmod2$.

		%\item 
		%$a,b< 2^{k-3}$ then $\binom{2^{k-1}-(a+b)}{2^{k-2}-a}\equiv 0\pmod 2$.
		%\item 
		%Consider $a,b\geq 2^{k-3}$. Without loss of generality we take $a=2^{k-3}, b=2^{k-3}+b'$ such that $b'< 2^{k-3}$. Then 
		%\begin{align*}
		%\binom{2^{k-1}-(a+b)}{2^{k-2}-a}&=\binom{2^{k-2}-b'}{2^{k-3}}.
		%\end{align*}
		%\item 
		%If $a,b\geq 2^{k-3}$, take $a=2^{k-3}+a'$ and $b=2^{k-3}+b'$. Then 
		%\begin{align*}
		%\binom{2^{k-1}-(a+b)}{2^{k-2}-a}&=\binom{2^{k-2}-(a'+b')}{2^{k-3}-a'}.
		%\end{align*}
		%Here $a',b'< 2^{k-3}$. We write $a'=2^{k-4}+a''$, $b'=2^{k-4}+b''$ and use the previous step.
		%\item 
		%If $a=2^{k-3}$ and $b<2^{k-3}$, then $\binom{2^{k-1}-(a+b)}{2^{k-2}-a}\equiv 0\pmod 2$.
		%\item 
		%If $a>2^{k-3}$ and $b<2^{k-3}$ ??

\end{prop}
\begin{proof}
The proof is immediate from Lucas Theorem \ref{lucas}.
\end{proof}

\begin{prop}\label{binomind} 
	Suppose $m$ is a positive integer. Then $\binom{m}{2^i} \equiv 0 \pmod 2 $ for $i=0,1,2,\ldots,k$, if and only if,  $2^{k+1} \mid m$.
\end{prop}
\begin{proof}
	The if part follows from Theorem \ref{lucas}. For the only if part, we prove the result by indution on $k$. The $k=0$ case is trivial.
	Suppose the statement holds for  $k=l-1$ i.e. if $\binom{m}{2^i} \equiv 0 \pmod 2$ for $i=0,1,\ldots,l-1$ then $2^l \mid m$.
  Now if moreover $\binom{m}{2^l} \equiv 0 \pmod 2$, then using Lucas Theorem \ref{lucas} we obtain that the $(l+1)^{th}$ digit in the binary expansion (counting from the right end) of $m$ should be $0$. Hence we get $2^{l+1}\mid m$. 
\end{proof}

\section{Review of Stiefel Whitney Classes} \label{swc.section}
\subsection{Group Cohomology}

For $R$ a ring, write $H^*(G,R)$ for the usual group cohomology ring, with $R$ regarded as a trivial $G$-module.

If $\varphi: G_1 \to G_2$ is a group homomorphism, we have an induced map $\varphi^*: H^*(G_2,R) \to H^*(G_1,R)$ on cohomology.

Write 
\begin{equation} \label{kappa}
\kappa : H^*(G,\Z) \to H^*(G,\Z/2\Z) 
\end{equation}
for the coefficient map of cohomology.

%Let us abbreviate $H^*(G,\Z/2\Z)$ as `$H^*(G)$' henceforth.

\subsection{Characteristic Classes}
%For a complex orthogonal representation $\pi$ of $G$ we write
%$$w_i'(\pi)=w_i(\pi_0).$$ 
%For convenience we write $w_i(\pi)$ in place of $w_i'(\pi)$.
%
%Let $G$ be a finite group. We have the map $\kappa:H^{2i}(G,\Z)\to H^{2i}(G,\Z_2)$,such that $\kappa(c^i(\pi))=w_{2i}(\pi_{\R})$.
%
%From \cite[Section 17.2, Exercise 19]{dummit} we obtain the following detection result:
%\begin{prop}\label{sylow}
%For a finite group $G$, the map $\res^*:H^*(G,\Z/p\Z)\to H^*(P,\Z/p\Z)$ is an injection, where $P$ is a $p$-Sylow  subgroup of $G$. 
%\end{prop}

In this section we review the theory of characteristic classes of representations of a finite group $G$.  Our reference is \cite{guna}.
Associated to complex representations $\pi$ of $G$ are cohomology classes $c_i(\pi) \in H^{2i}(G,\Z)$, for $0 \leq i \leq \deg \pi$, called Chern classes.   
We have $c_0(\pi)=1$.  The first Chern class, applied to linear characters, gives an isomorphism 
\beq
c_1: \mathrm{Hom}(G,S^1) \overset{\sim}{\to} H^2(G,\Z).
\eeq
This extends to arbitrary complex representations $\pi$ by $c_1(\pi)=c_1(\det \pi)$.

%Associated to real representations $(\rho,V)$ of $G$ are cohomology classes $w^{\R}_i(\rho) \in H^i(G,\Z/2\Z)$, for $0 \leq i \leq \deg \rho$, called Stiefel Whitney (SW) classes.
%The total Stifel Whitney classes (SWC) is then
%\beq
%w^\R(\rho)=w_0^\R(\rho) + \cdots + w_{d}^\R(\rho) \in \bigoplus_{i=0}^d H^i(G,\Z/2\Z),
%\eeq
%where $d=\deg \rho$.
%
%The classical approach (see \cite[Section 2.6]{benson})  is to associate to $(\rho,V)$  a certain real vector bundle $\mc V$ over the classifying space $BG$. The Stifel Whitney class of $\rho$ is then defined to be  the Stifel Whitney class of $\mc V$ as developed in \cite{mch}.
%Alternatively, the paper \cite{guna} gives an axiomatic characterization of Stiefel Whitney classs of real representations.  With the following modification we obtain a theory of Stiefel Whitney classs of \emph{orthogonal complex} representations.

\begin{defn} If $\pi$ is an orthogonal complex representation  of $G$, put
	\beq
	w^\C_i(\pi)=w^\R_i(\pi_0),
	\eeq
	for $0 \leq i \leq \deg \pi$.
\end{defn}

Thus if $\rho$ is a real representation, we have $w^\R(\rho)=w^\C(\rho \otimes_\R \C)$.

\begin{lemma}\label{kap} If $\pi$ is a complex representation, then 
	\beq
	w^\C(S(\pi))=\kappa(c(\pi)).
	\eeq
\end{lemma}

\begin{proof}
	We have
	\beq
	\begin{split}
		w^\C(S(\pi)) &= w^\R(S(\pi)_0) \\
		&= w^\R(\pi_\R) \\
		&= \kappa(c(\pi)), \\
	\end{split}
	\eeq
	the last equality by  \cite[Problem 14-B]{mch}.
\end{proof}

Henceforth all representations will be complex representations, and we will drop the superscript `$\C$' from $w^\C$.

Let $\pi$ be an orthogonal (complex) representation.  Again one has $w_0(\pi)=1$, and the first Stiefel Whitney class, applied to linear characters $G \to \{\pm 1\} \cong \Z/2\Z$, is the well-known isomorphism
\beq
w_1: \mathrm{Hom}(G, \Z/2\Z)  \overset{\sim}{\to} H^1(G, \ztwo).
\eeq
Therefore $w_1(\pi)=\det\pi$.

\begin{defn}\label{achiral}
A real representation $\pi$ of $G$ is called achiral if $w_1(\pi)=0$. It is called chiral otherwise. 
\end{defn}

%The Stiefel Whitney classes are natural in the sense that if $\varphi: G_1 \to G_2$ is a group homomorphism, and
%$\pi$ is an orthogonal representation of $G_2$, then we have
%\begin{equation} \label{funct.w}
%\varphi^*(w(\rho))=w(\rho \circ \varphi).
%\end{equation}
%Similarly for Chern classes.
%
%
%
%
%
%
%If $\pi$ decomposes into a direct sum $\pi_1 \oplus \pi_2$ of orthogonal representations, then
%\begin{equation} \label{cup}
%w(\pi)=w(\pi_1) \cup w(\pi_2).
%\end{equation}
%This is called ``additivity", and the analogous holds for Chern classes.

\begin{prop} \label{spin.swc} Let $G$ be a finite group, and $\pi$ an orthogonal representation of $G$.  Then $\varphi$ is spinorial iff
	\beq
	w_2(\pi)=w_1(\pi) \cup w_1(\pi).
	\eeq
\end{prop}

\begin{proof} See, for instance, \cite{guna}.
	%\cite[Proposition $6.1$]{Ganguly}.
\end{proof}

%We will refer to this throughout this paper as the ``$w_2=w_1^2$ criterion".  
%For instance, if $\varphi$ takes values in $\SO(V)$, then it is spinorial iff $w_2(\varphi)=0$. 

\subsection{Detection}\label{detect}

Let $G'$ be a subgroup of a finite group $G$. 
%Write $\iota: G' \to G$ for the inclusion. 
We say that $G'$ \emph{detects} the (mod $2$) cohomology of $G$, provided that the restriction map
\beq
\res: H^*(G, \ztwo) \to H^*(G',\ztwo)
\eeq
is injective. 
%Detecting subgroups also detect spinoriality:
%\begin{proof%\begin{prop} \label{detect.spin} Suppose that $G'$ detects the cohomology of $G$. If $\pi$ is an orthogonal representation of $G$, then $\pi$ is spinorial iff
%	$\pi|_{G'}$ is spinorial.
%\end{prop}} Recall the $w_2=w_1^2$ criterion from Proposition \ref{spin.swc}. By hypothesis and the naturality of SWCs, we have
%	\beq
%	w_2(\pi)=w_1(\pi)^2 \hspace{.7 cm} \Leftrightarrow \hspace{.7 cm} w_2(\pi|_{G'})=w_1(\pi|_{G'})^2.
%	\eeq
%\end{proof}
A special case of \cite[Section 17.2, Exercise 19]{Dummit} gives:

\begin{prop} \label{sylow.spin} If $P$ is a $2$-Sylow subgroup of $G$, then $P$ detects the (mod $2$) cohomology of $G$.
\end{prop}

\begin{prop}\label{sylow1}
Let $P$ denote the $2$-Sylow subgroup of $C_n$. Then 
$$H^*(C_n,\Z/2\Z)=H^*(P,\Z/2\Z).$$
\end{prop}
\begin{proof}
This follows from the fact that $C_n=P\times H$, where $H$ is a cyclic group of odd order. 
\end{proof}

The following theorem is found in \cite[theorem $4.4$, page no. $227$]{adem}: 
%\cite[Section 8] {quillen}:

\begin{thm}\label{thm3} If $q$ is odd, then the subgroup $D$ of diagonal matrices in $G=\GL_2(\mb F_q)$ detects the (mod $2$) cohomology of $G$.
\end{thm}

\bigskip
%In the next section we describe the well-known cohomology ring $H^*(A)$.  
%The cohomology $H^*(G)$ is a certain subspace of the (Weyl-invariants of) $H^*(A)$.
For an orthogonal representation $\pi$ of $G$, we will often write `$w_i(\pi)$' as an abbreviation for 
\beq
\iota^*(w_i(\pi))=w_i(\pi|_D).
\eeq

\section{Total Stiefel Whitney Classes For Real Representations of $\GL_2(\F_q)$}\label{sectsw}

\subsection{Results for the Bicyclic Group}

We first consider the bicyclic group $C_n\times C_n$, where $n\equiv 0\pmod 4$. One has
$$H^*(C_n\times C_n,\ztwo)=\dfrac{\ztwo[s_1,t_1,s_2,t_2]}{\left(s_1^2,s_2^2\right)},$$
where $s_1=w_1(\sgn\otimes \mathbb{1})$ and $s_2=w_1(\mathbb{1}\otimes \sgn)$, $t_1=w_2((\chi^1\otimes 1)_{\mathbb{R}})$ and $t_2=w_2((1\otimes\chi^1)_{\mathbb{R}})$. 
%(see Section \ref{3} and Equation \eqref{kap1}). 
%
%Let $\bar{S}=\{(j_1,j_2)\mid j_1, j_2 \in \Z, -n/2 \leq j_1 \leq n/2-1 , -n/2 \leq j_2 \leq n/2-1  \} - \{(0,0), (-n/2,0),(0,-n/2),(-n/2,-n/2)\}$. let S be subset of $\bar{S}$ such that it contains exactly 1 element from each of the following pairs 
%\begin{enumerate}
%	\item $(i,j)$ and $(-i,-j)$ where $-n/2+1 \leq i,j \leq n/2 -1$,
%	\item $(-n/2, i)$ and $(-n/2, -i)$ where $1 \leq i\leq n/2-1$,
%	\item  $(i,-n/2)$ and $(-i,-n/2)$ where $1 \leq i \leq n/2 -1$.
%\end{enumerate}
Let $$\bar{S}=\Z/n\Z\times \Z/n\Z-\{(0,0), (n/2,0), (0,n/2), (n/2,n/2)\}$$ and 
\begin{equation}\label{s}
S=\bar{S}/\sim,\,\quad \text{where $\sim$ denotes the relation}\, (i,j)\sim (n-i,n-j).
\end{equation}
Consider a real representation $\pi$ of $C_n\times C_n$ of the form 
\begin{equation}\label{eqpi}
\pi=m_0\mathbb{1}\oplus m_1(\sgn\otimes\mathbb{1})\oplus m_2(\mathbb{1}\otimes \sgn)\oplus m_3(\sgn\otimes \sgn)\bigoplus\limits_{(j_1,j_2)\in S} M_{j_1j_2}\left((\chi_{j_1}\otimes \chi_{j_2})_{\mathbb{R}}\right),
\end{equation}
where $m_0, m_1, m_2, m_3, M_{j_1j_2}\in \Z_{\geq 0}$.
%To calculate the total SW classes we use the following identities:
%
%\begin{align*}
%w_1(\pi\otimes \pi')&=\deg \pi'\cdot w_1(\pi)+\deg \pi\cdot w_1(\pi')\\
%w_2(\pi\otimes \pi')&=\deg \pi'\cdot w_2(\pi)+\binom{\deg \pi'}{2}w_1(\pi)\cup w_1(\pi)
%+(\deg\pi\deg\pi'-1)w_1(\pi)\otimes w_1(\pi')\\
%&+\binom{\deg \pi}{2}w_1(\pi')\cup w_1(\pi')+\deg \pi\cdot w_2(\pi')
%\end{align*}
%
%We have $w_1(\sgn\otimes\sgn)=s_1+s_2$. 
We compute
\begin{equation}
w(\pi)=(1+s_1)^{m_1}\cdot (1+s_2)^{m_2}\cdot (1+s_1+s_2)^{m_3}\cdot \prod\limits_{(j_1,j_2)\in S}(1+j_1t_1+j_2t_2)^{M_{j_1j_2}}.\\
\end{equation}
Let us define
\begin{equation}\label{eq5}
S_{xy}=\{(j_1,j_2)\in S\mid j_1\equiv x\pmod 2, j_2\equiv y\pmod 2\},
\end{equation}
where $x,y \in \{0,1\}$.
%	\item $S_{10}=\{(j_1,j_2)\in S\mid j_1\equiv 1\pmod 2, j_2\equiv 0\pmod 2\}$,
%	\item $S_{01}=\{(j_1,j_2)\in S\mid j_1\equiv 0\pmod 2, j_2\equiv 1\pmod 2\}$,
%	\item $S_{11}=\{(j_1,j_2)\in S\mid j_1\equiv 1\pmod 2, j_2\equiv 1\pmod 2\}$
We put 
\begin{equation}\label{eq6}
m_{xy}=\sum\limits_{(j_1,j_2)\in S_{xy}}M_{j_1j_2}\quad \text{and}\quad m'= m_0+m_1+m_2+m_3.
\end{equation}
%Note that the restriction of the representation $\pi$ (as in \eqref{eqpi}) to $C_2\times C_2$  gives
%\begin{equation}
%\pi\mid_{C_2\times C_2}=(m_0+m_1+m_2+m_3+2m_{00})\mathbb{ 1}\oplus 2m_{10}(\sgn\otimes \mathbb{ 1})\oplus 2m_{01}(\mathbb{ 1}\otimes \sgn)\oplus 2m_{11}(\sgn\otimes \sgn).
%\end{equation}
%Therefore we have 
%\begin{itemize}
%\item 
%$m_{10}=\frac{1}{2}\dim\mathrm{Hom}(\pi\mid_{C_2\times C_2}, \sgn\otimes \mathbb{ 1})$,
%\item 
%$m_{01}=\frac{1}{2}\dim\mathrm{Hom}(\pi\mid_{C_2\times C_2}, \mathbb{ 1}\otimes \sgn)$,	
%\item 
%$m_{11}=\frac{1}{2}\dim\mathrm{Hom}(\pi\mid_{C_2\times C_2}, \sgn\otimes \sgn)$.
%\end{itemize}	

%\begin{lemma} We have
%	\begin{align*}
%	w(\pi)&=\left(1+(m_1+m_3)s_1+(m_2+m_3)s_2+(m_1m_2+m_2m_3+m_3m_1)s_1s_2\right)\\
%	&\cdot (1+t_1)^{m_{10}}\cdot (1+t_2)^{m_{01}}\cdot (1+t_1+t_2)^{m_{11}}.
%	\end{align*}
%\end{lemma}
%
%\begin{proof}
%	Since $s_i^2=0$, we obtain the expression
%	\begin{align*}
%	w(\pi)&=(1+m_1s_1)\cdot (1+m_2s_2)\cdot (1+m_3s_1+m_3s_2)\cdot \prod\limits_{(j_1,j_2)\in S} (1+j_1t_1+j_2t_2)^{m_{j_1j_2}}
%	\end{align*}
%	
%	\begin{align*}
%	w(\pi)&=(1+m_1s_1)\cdot (1+m_2s_2)\cdot (1+m_3s_1+m_3s_2)\cdot (1+t_1)^{m_{10}}\cdot (1+t_2)^{m_{01}}\cdot (1+t_1+t_2)^{m_{11}}\\
%	&=\left(1+(m_1+m_3)s_1+(m_2+m_3)s_2+(m_1m_2+m_2m_3+m_3m_1)s_1s_2\right)\\
%	&\cdot (1+t_1)^{m_{10}}\cdot (1+t_2)^{m_{01}}\cdot (1+t_1+t_2)^{m_{11}}.
%	\end{align*}
%\end{proof}

%Write $\delta_1=\det(\pi|_{C_n\times 1})$ and   $\delta_2=\det(\pi|_{1\times C_n})$. Note that $\delta_1=(m_1+m_3)w_1(\sgn\otimes 1)=$. Similarly $\delta_2=(m_2+m_3)w_1(1\otimes \sgn)$.

\begin{lemma}\label{character.formula}
	For $(x,y)\in \{0,1\}^2\setminus (0,0)$, we have
	$$m_{xy}=\frac{1}{8}\left(\chi_{\varphi}(0,0)+(-1)^x\chi_{\varphi}(n/2,0)+(-1)^y\chi_{\varphi}(0,n/2)+(-1)^{x+y}\chi_{\varphi}(n/2,n/2) \right).$$
	%	\begin{enumerate}
	%		\item 
	%		$m'+2m_{00}=\frac{1}{4}(\chi_{\pi}(0,0)+\chi_{\pi}(n/2,0)+\chi_{\pi}(0,n/2)+\chi_{\pi}(n/2,n/2))$,
	%		\item 
	%		$m_{10}=\frac{1}{8}(\chi_{\pi}(0,0)-\chi_{\pi}(n/2,0)+\chi_{\pi}(0,n/2)-\chi_{\pi}(n/2,n/2))$,
	%		\item 
	%		$m_{01}=\frac{1}{8}(\chi_{\pi}(0,0)+\chi_{\pi}(n/2,0)-\chi_{\pi}(0,n/2)-\chi_{\pi}(n/2,n/2))$,
	%		\item 
	%		$m_{11}=\frac{1}{8}(\chi_{\pi}(0,0)-\chi_{\pi}(n/2,0)-\chi_{\pi}(0,n/2)+\chi_{\pi}(n/2,n/2))$.
	%%		\item 
	%%		$m_1m_2+m_2m_3+m_3m_1=(\delta_{10}\cdot \delta_{01}+\deg \pi+\deg \pi^{C_n\times C_n}+\delta_{10}+\delta_{01})$.
	%	\end{enumerate}
\end{lemma}

\begin{proof}
	We get the result by solving the following equations:
	\begin{align*}
	\chi_{\pi}(0,0)&=m'+2(m_{00}+m_{10}+m_{01}+m_{11}),\\
	\chi_{\pi}(n/2,0)&= m'+2(m_{00}-m_{10}+m_{01}-m_{11}),\\
	\chi_{\pi}(0,n/2)&= m'+2(m_{00}+m_{10}-m_{01}-m_{11}),\\
	\chi_{\pi}(n/2,n/2)&=m'+2(m_{00}-m_{10}-m_{01}+m_{11}).
	\end{align*}
\end{proof}

%Then we write the total Stiefel Whitney class as follows:
We set 

\[
\delta_1=
\begin{cases}
0,\quad \text{if $\det(\pi\mid_{C_n\times 1})=\mathbb{ 1}$},\\
1,\quad\text{if $\det(\pi\mid_{C_n\times 1})=\sgn$}.
\end{cases}
\]
Similarly we define $\delta_2\in \{0,1\}$ depending on $\det(\pi\mid_{1\times C_n})$ . For a representation $\pi$ of $G$ we write $\pi^G$ to denote the fixed space of $G$.

\begin{lemma}\label{lem0}
	Consider a real representation $\pi$ of $C_n\times C_n$ of the form \ref{eqpi}. Then we have 
	\begin{enumerate}
		\item 
		$w_1(\pi)=\delta_1s_1+\delta_2s_2$,
		\item
		$w_2(\pi)=(\delta_{1}\cdot \delta_{2}+\dim \pi+\dim \pi^{C_n\times C_n}+\delta_{1}+\delta_{2})s_1s_2+(m_{10}+m_{11})t_1+(m_{01}+m_{11})t_2$,
		\item
		Moreover, if $m_1m_2+m_2m_3+m_3m_1\equiv 0\pmod 2$ then
		\begin{equation}\label{wpi}
		w(\pi)=(1+\delta_{1}s_1 + \delta_{2}s_2)(1+t_1)^{m_{10}}\cdot (1+t_2)^{m_{01}}\cdot (1+t_1+t_2)^{m_{11}}.
		\end{equation}
		%and $m_{10}+m_{11}\equiv 0\pmod 2, m_{01}+m_{11}\equiv 0\pmod 2.$
	\end{enumerate} 
\end{lemma}

\begin{proof}
	Since $s_i^2=0$, the expression for $w(\pi)$ reduces to
	\begin{align*}
	w(\pi)&=(1+m_1s_1)\cdot (1+m_2s_2)\cdot (1+m_3s_1+m_3s_2)\cdot \prod\limits_{(j_1,j_2)\in S} (1+j_1t_1+j_2t_2)^{M_{j_1j_2}},
	\end{align*}
	where the set $S$ is as mentioned in \eqref{s}. Using Equations \eqref{eq5} and \eqref{eq6} one calculates 
	%\begin{align*}
	%\prod\limits_{(j_1,j_2)\in S_{xy}} (1+j_1t_1+j_2t_2)^{M_{j_1j_2}} =(1+xt_1+yt_2)^{m_{xy}},
	%\end{align*}
	%where $x,y\in\{0,1\}$. Therefore
	\begin{equation*}
	\prod\limits_{(j_1,j_2)\in S} (1+j_1t_1+j_2t_2)^{M_{j_1j_2}}=(1+t_1)^{m_{10}}\cdot (1+t_2)^{m_{01}}\cdot (1+t_1+t_2)^{m_{11}}.
	\end{equation*}
	%We have 
	%\begin{align*}
	%(1+m_1s_1)\cdot (1+m_2s_2)\cdot (1+m_3s_1+m_3s_2)&= 1+(m_1+m_3)s_1+(m_2+m_3)s_2\\
	%&+(m_1m_2+m_2m_3+m_3m_1)s_1s_2.
	%\end{align*}
	We have
	\begin{equation}\label{eq7}
	\begin{split}
	w(\pi)=&\left(1+(m_1+m_3)s_1+(m_2+m_3)s_2+(m_1m_2+m_2m_3+m_3m_1)s_1s_2\right)\\
	&\cdot(1+t_1)^{m_{10}}\cdot (1+t_2)^{m_{01}}\cdot (1+t_1+t_2)^{m_{11}}.
	\end{split}
	\end{equation}

	%\begin{align*}
	%w(\pi)&=(1+m_1s_1)\cdot (1+m_2s_2)\cdot (1+m_3s_1+m_3s_2)\cdot (1+t_1)^{m_{10}}\cdot (1+t_2)^{m_{01}}\cdot (1+t_1+t_2)^{m_{11}}\\
	%&=\left(1+(m_1+m_3)s_1+(m_2+m_3)s_2+(m_1m_2+m_2m_3+m_3m_1)s_1s_2\right)\\
	%&\cdot (1+t_1)^{m_{10}}\cdot (1+t_2)^{m_{01}}\cdot (1+t_1+t_2)^{m_{11}}.
	%\end{align*}	
	This gives $w_1(\pi)=(m_1+m_3)s_1+(m_2+m_3)s_2$ and
	\begin{equation}\label{w2}
	w_2(\pi)=(m_1m_2+m_2m_3+m_3m_1)s_1s_2+(m_{10}+m_{11})t_1+(m_{01}+m_{11})t_2.
	\end{equation}
	We have
	\begin{enumerate}
		\item 
		$\delta_i\equiv m_i+m_3\pmod 2$ for $i\in\{1,2\}$,
		\item 
		$\dim \pi^{C_n\times C_n}=m_0$,
		\item 
		$\dim \pi\equiv m_0+m_1+m_2+m_3\pmod 2$.
	\end{enumerate}
	
	Using these facts one computes
	%\begin{align*}
	%\delta_{1}\cdot \delta_{2}+\dim \pi+\dim \pi^{C_n\times C_n}+\delta_{1}+\delta_{2}
	%&\equiv (m_1+m_3)\cdot (m_2+m_3)+m_0+m_1+m_2\\
	%&+m_3+m_0+m_1+m_2\pmod 2\\
	%&\equiv m_1m_2+m_1m_3+m_2m_3+m_3^2+m_3\pmod 2\\
	%&\equiv m_1m_2+m_1m_3+m_2m_3\pmod 2.
	%\end{align*}
	\begin{equation*}
	\delta_{1}\cdot \delta_{2}+\dim \pi+\dim \pi^{C_n\times C_n}+\delta_{1}+\delta_{2}\equiv  m_1m_2+m_1m_3+m_2m_3\pmod 2.
	\end{equation*}
	If $\pi$ is achiral then $\delta_i\equiv 0\pmod 2$ for $i\in\{1,2\}$. Also if $(m_1m_2+m_2m_3+m_3m_1)\equiv 0\pmod 2$ then $w(\pi)$ takes the required form. 
	
	%	\begin{equation}\label{eq8}
	%	m_1+m_3\equiv 0\pmod 2, m_2+m_3\equiv 0\pmod 2.
	%	\end{equation}
	%	and
	
	%	\begin{equation}\label{eq9}
	%	\begin{split}
	%	& (m_1m_2+m_2m_3+m_3m_1)\equiv 0\pmod 2, \\
	%	& m_{10}+m_{11}\equiv 0\pmod 2, m_{01}+m_{11}\equiv 0\pmod 2.
	%	\end{split}
	%     \end{equation}   

	%Therefore from Equations \eqref{eq7}, \eqref{eq8} and \eqref{eq9} we have the result.
\end{proof}

Next we consider the case when $n\equiv 2\pmod 4$. We have 
$$H^*(C_n\times C_n,\ztwo)=\ztwo[v_1,v_2],$$
where $v_1=w_1(\sgn\otimes \mathbb{1})$ and $v_2=w_1(\mathbb{1}\otimes \sgn)$.
Since $C_2 \times C_2$ is the 2 sylow subgroup of $C_n \times C_n$ it %is enough to calculate $w_i(\pi\mid_{C_2 \times C_2})$.
detects the cohomology of $C_n\times C_n$.

Consider a real representation $\varphi$ of $C_n\times C_n$. Let
\begin{equation}\label{eq4}
\varphi \mid_{C_2 \times C_2} =  m_{00}\mathbb{1} \oplus m_{10}(\sgn \otimes \mathbb{1}) \oplus m_{01}( \mathbb{1} \otimes \sgn) \oplus m_{11}(\sgn \otimes \sgn).
\end{equation}

\begin{lemma}\label{character.formula.C2}
	For $x,y\in \{0,1\}$, we have
	$$m_{xy}=\frac{1}{4}\left(\chi_{\varphi}(0,0)+(-1)^x\chi_{\varphi}(n/2,0)+(-1)^y\chi_{\varphi}(0,n/2)+(-1)^{x+y}\chi_{\varphi}(n/2,n/2) \right).$$
	%	\begin{enumerate}
	%		\item 
	%		$m_{00}=\frac{1}{4}(\chi_{\varphi}(0,0)+\chi_{\varphi}(n/2,0)+\chi_{\varphi}(0,n/2)+\chi_{\varphi}(n/2,n/2))$.
	%		\item 
	%		$m_{10}=\frac{1}{4}(\chi_{\varphi}(0,0)-\chi_{\varphi}(n/2,0)+\chi_{\varphi}(0,n/2)-\chi_{\varphi}(n/2,n/2))$.
	%		\item 
	%		$m_{01}=\frac{1}{4}(\chi_{\varphi}(0,0)+\chi_{\varphi}(n/2,0)-\chi_{\varphi}(0,n/2)-\chi_{\varphi}(n/2,n/2))$.
	%		\item 
	%		$m_{11}=\frac{1}{4}(\chi_{\varphi}(0,0)-\chi_{\varphi}(n/2,0)-\chi_{\varphi}(0,n/2)+\chi_{\varphi}(n/2,n/2))$.
	%%		\item 
	%%		$m_1m_2+m_2m_3+m_3m_1=(\delta_{10}\cdot \delta_{01}+\deg \pi+\deg \pi^{C_n\times C_n}+\delta_{10}+\delta_{01})$.
	%	\end{enumerate}
\end{lemma}
\begin{proof}
	The proof is similar to that of Lemma \ref{character.formula}.	
	%We get the result by solving the following equations:	
	%	\begin{align*}
	%	\chi_{\varphi}(0,0)&=m_{00}+m_{10}+m_{01}+m_{11},\\
	%	\chi_{\varphi}(n/2,0)&= m_{00}-m_{10}+m_{01}-m_{11},\\
	%	\chi_{\varphi}(0,n/2)&= m_{00}+m_{10}-m_{01}-m_{11},\\
	%	\chi_{\varphi}(n/2,n/2)&=m_{00}-m_{10}-m_{01}+m_{11}.
	%	\end{align*}
\end{proof}

\begin{theorem}\label{q3tot}
	The total Stiefel Whitney class of $\varphi$ is 
	$$w(\varphi) = (1+v_1)^{m_{10}}(1+v_2)^{m_{01}}(1+v_1+v_2)^{m_{11}}.  $$
\end{theorem}
\begin{proof}
	The	proof follows immediately from Equation \eqref{eq4}.
\end{proof}

\subsection{Results for the Group $\GL_2(\F_q)$}

The cohomology of $\mathrm{GL}_2(\mathbb{F}_q)$ is detected by the diagonal subgroup $D\cong C_{q-1}\times C_{q-1}$ (see Theorem \ref{thm3}). 
Let $\pi$ be a real representation of $\GL_2(\F_q)$ and $q\equiv 1\pmod 4$. Then the expressions appeared in Lemma \ref{character.formula} for $\pi\mid_D$ become
\begin{enumerate}
	%	\item 
	%	$m'+2m_{00}=\frac{1}{4}(\chi_{\pi}(0,0)+\chi_{\pi}(n/2,0)+\chi_{\pi}(0,n/2)+\chi_{\pi}(n/2,n/2))$
	%	
	%\item 
	%	$m_{10}=\frac{1}{8}(\chi_{\pi}(z_1)-\chi_{\pi}(z_{-1}))$
	\item 
	$m_{10}=m_{01}=\frac{1}{8}(\chi_{\pi}(z_1)-\chi_{\pi}(z_{-1}))$
	\item 
	$m_{11}=\frac{1}{8}(\chi_{\pi}(z_1)-2\chi_{\pi}(t_{(1,-1)})+\chi_{\pi}(z_{-1}))$,
	%	\item 
	%	$m_1m_2+m_2m_3+m_3m_1=(\delta_{10}\cdot \delta_{01}+\deg \pi+\deg \pi^{C_n\times C_n}+\delta_{10}+\delta_{01})$.
\end{enumerate}
where $z_x=\mat x00x$ and $t_{(x,y)}=\mat x00y$. 
%Note that these terms appeared in Lemma \ref{character.formula}. 
%Similarly for a representation $\varphi$ of $\GL_2(\F_q)$, $q\equiv 3\pmod 4$, we take
%\begin{enumerate}
%	%	\item 
%	%	$m'+2m_{00}=\frac{1}{4}(\chi_{\pi}(0,0)+\chi_{\pi}(n/2,0)+\chi_{\pi}(0,n/2)+\chi_{\pi}(n/2,n/2))$
%	%	
%	%\item 
%	%	$m_{10}=\frac{1}{8}(\chi_{\pi}(z_1)-\chi_{\pi}(z_{-1}))$
%	\item 
%	$m_{10}=m_{01}=\frac{1}{4}(\chi_{\varphi}(z_1)-\chi_{\varphi}(z_{-1}))$
%	\item 
%	$m_{11}=\frac{1}{4}(\chi_{\varphi}(z_1)-2\chi_{\varphi}(t_{(1,-1)})+\chi_{\varphi}(z_{-1}))$.
%	%	\item 
%	%	$m_1m_2+m_2m_3+m_3m_1=(\delta_{10}\cdot \delta_{01}+\deg \pi+\deg \pi^{C_n\times C_n}+\delta_{10}+\delta_{01})$.
%\end{enumerate}
We write $w_i(\pi)$ as an abbreviation for $\res(w_i(\pi))=w_i(\pi|_D).$
%\begin{theorem}
%Let $\pi$ be a real representation of $\GL_2(\F_q)$, $q\equiv 1\pmod 4$ such that $w_1(\pi)=w_2(\pi)=0$. Then the total Stiefel Whitney class of $\pi$ is 
%$$w(\pi)=(1+t_1)^{m_{10}}\cdot (1+t_2)^{m_{01}}\cdot (1+t_1+t_2)^{m_{11}}.$$
%\end{theorem}
%
%\begin{proof}
%The proof is immediate from Lemma \ref{lem0}.
%\end{proof}
Recall that for $q\equiv 1\pmod 4$ 
$$H^*(GL(2,\F_q)) = \ztwo[e_1,e_2,c_1,c_2]/(e_1^2=0, e_2^2=0), $$
where $e_1,e_2,c_1,c_2$ are elements of degree $1,3,2,4$ respectively (see Section \ref{intro}).

\begin{proof}[Proof of Theorem \ref{tsw}]
	
	By \cite[page 571]{quillen} we have an injection          $$\res:H^*(\GL_2(\F_q))\to H^*(D)^{S_2},$$
	where $S_2$ is the symmetric group of order 2 acting on $D$ by simply permuting the two diagonal entries.
	By degree considerations we have
	$$\res(e_1)=s_1+s_2,\quad \res(e_2)= s_1t_2+s_2t_1, \quad \res(c_1)=t_1+t_2 ,\quad \res(c_2)= t_1t_2.$$
	%Consider a real representation $\varphi$ of $\GL_2(\F_q)$.
	Using Lemma \ref{lem0} for the the representation $\varphi\mid_D$ one obtains $$w_2(\varphi)=(m_1m_2+m_2m_3+m_3m_1)s_1s_2+(m_{10}+m_{11})t_1+(m_{01}+m_{11})t_2.  $$
	But since $s_1s_2\notin H^*(\GL_2(\F_q))$ and $m_{01}=m_{10}$, we simply have $$w_2(\varphi)=(m_{10}+m_{11})(t_1+t_2). $$     
	This implies $m_1m_2+m_2m_3+m_3m_1 \equiv 0 \pmod 2$.
	Moreover  
	$$\delta=\delta_1=\delta_2= m_1+m_3 \pmod 2 =m_2+m_3 \pmod 2.$$
	Theorefore from Lemma \ref{lem0} we have 
	$$w(\varphi) = (1+\delta (s_1+s_2))((1+t_1)(1+t_2))^{m_{10}}(1+t_1+t_2)^{m_{11}} .$$
	
\end{proof}

We end this section with a discussion on total Stiefel Whitney classes of real representations of $\GL_2(\F_q)$, where $q\equiv 3\pmod 4$.
Consider the terms (following  \ref{character.formula.C2})
\begin{enumerate}
	%	\item 
	%	$m'+2m_{00}=\frac{1}{4}(\chi_{\pi}(0,0)+\chi_{\pi}(n/2,0)+\chi_{\pi}(0,n/2)+\chi_{\pi}(n/2,n/2))$
	%	
	%\item 
	%	$m_{10}=\frac{1}{8}(\chi_{\pi}(z_1)-\chi_{\pi}(z_{-1}))$
	\item 
	$m_{10}=m_{01}=\frac{1}{4}(\chi_{\pi}(z_1)-\chi_{\pi}(z_{-1}))$
	\item 
	$m_{11}=\frac{1}{4}(\chi_{\pi}(z_1)-2\chi_{\pi}(t_{(1,-1)})+\chi_{\pi}(z_{-1}))$.
	%	\item 
	%	$m_1m_2+m_2m_3+m_3m_1=(\delta_{10}\cdot \delta_{01}+\deg \pi+\deg \pi^{C_n\times C_n}+\delta_{10}+\delta_{01})$.
\end{enumerate}

\begin{lemma}\label{m10even}
	Let $\pi$ be a real representation of $GL_2(\F_q)$ where $q \equiv 3 \pmod 4$. Then the quantity
	$m_{10} = (\dim \pi - \chi_{\pi}(z_{-1}))/4$ is even.
\end{lemma}

\begin{proof}
	The result can be verified easily for OIRs of $\GL_2(\F_q)$.	
	Since $m_{10}$ is additive with respect to direct sum it follows for any real representation $\pi$ of $GL_2(\F_q)$.
\end{proof}

From [Totaro, page $30$, Theorem $2.23$] we have for $q\equiv 1\pmod 4$ 
$$H^*(GL(2,\F_q)) = \ztwo[e_1,e_2,c_1,c_2]/(e_1^2=c_{1}, e_2^2=c_1c_2), $$
where $e_1,e_2,c_1,c_2$ are elements of degree $1,3,2,4$ respectively.
We have 
$$\res(e_1)=v_1+v_2,\quad \res(c_1)=v_1^2+v_2^2,\quad \res(e_2)=(v_1+v_2)v_1v_2,\quad \res(c_2)=v_1^2v_2^2.$$

\begin{theorem}\label{tsw2}
	Let $\pi$ be a real representation of $\GL_2(\F_q)$, where $q\equiv 3\pmod 4$.  Then the total Stiefel Whitney class of $\pi$ is 
	\begin{equation*}
	w(\pi) = (1+c_1+c_2)^{m_{10}/2}(1+e_1)^{m_{11}}.
	\end{equation*}
\end{theorem}

\begin{proof}
	The proof follows from Lemma \ref{m10even} and Theorem \ref{q3tot} immediately.
\end{proof}

\section{Obstruction Classes for Real Representations of $C_n\times C_n$}\label{bicyclic}

We aim to calculate the obstruction classes of real representations of $C_n\times C_n$. Throughout the section we assume (unless otherwise mentioned) $\pi$ to be a real representation of $C_n\times C_n$ of the form as in \ref{eqpi}. Also assume that $w_1(\pi)=0$ and $m_1m_2+m_2m_3+m_3m_1\equiv 0\pmod 2$.

\begin{lemma}\label{lem1}
	We have
		\begin{align*}
		w_4(\pi)=&\binom{m_{10}+m_{11}}{2}t_1^2+\binom{m_{01}+m_{11}}{2}t_2^2\\
		&+\left(\binom{m_{01}+m_{11}}{2} + \binom{m_{10}+m_{11}}{2}+\binom{m_{01}+m_{10}}{2}\right)t_1t_2.
		\end{align*}
%		\item
%		If $w_i(\pi)=0$, for $1\leq i\leq 4$, then
%		\begin{equation*}
%		w_8(\pi)=\binom{m_{10}+m_{11}}{4}t_1^4+\binom{m_{01}+m_{11}}{4}t_2^4+\left(\binom{m_{10}+m_{01}}{4}+\binom{m_{10}+m_{11}}{4}+\binom{m_{01}+m_{11}}{4}\right)t_1^2t_2^2.
%		\end{equation*}
	\end{lemma}

\begin{proof}
	
	We collect the degree $4$ terms from $w(\pi)$ as in Equation \eqref{wpi} and obtain
	\begin{align*}
	w_4(\pi)&=\binom{m_{10}+m_{11}}{2}t_1^2+\binom{m_{01}+m_{11}}{2}t_2^2
	+(m_{01}m_{11}+ m_{10}m_{11} +m_{01}m_{10})t_1t_2.\\
	\end{align*}	
 By Vandermonde identity \ref{van} we have $$m_{01}m_{11}+ m_{10}m_{11} +m_{01}m_{10} \equiv \binom{m_{01}+m_{11}}{2} + \binom{m_{10}+m_{11}}{2}+\binom{m_{01}+m_{10}}{2} \pmod 2 .$$
\end{proof}

%In general we conjecture
%\begin{equation*}
%w_{2^k}(\pi)=\binom{m_{10}+m_{11}}{2^{k-1}}t_1^{2^{k-1}}+\binom{m_{01}+m_{11}}{2^{k-1}}t_2^{2^{k-1}}+\left(\binom{m_{10}+m_{01}}{2^{k-1}}+\binom{m_{10}+m_{11}}{2^{k-1}}+\binom{m_{01}+m_{11}}{2^{k-1}}\right)t_1^{2^{k-1}}t_2^{2^{k-1}}.
%\end{equation*}

\begin{lemma} \label{first2} 
	The coefficient of $t_1^{2^{k-1}}$ in $w_{2^k}(\pi)$ is $\binom{m_{10} + m_{11}}{2^{{k-1}}} $ and that of $t_2^{2^{k-1}}$ is $\binom{m_{01} + m_{11}}{2^{{k-1}}} .$
\end{lemma}
\begin{proof}
	From Lemma \ref{lem0} we have $w(\pi) = (1+t_1)^{m_{10}}(1+t_2)^{m_{01}}(1+t_1+t_2)^{m_{11}}$.  Note that the term $t_1^{2^{k-1}}$ is obtained from the factor $(1+t_1)^{m_{10}}(1+t_1+t_2)^{m_{11}}$. Therefore the coefficient of $t_1^{2^{k-1}}$ is equal to the number of ways to choose $2^{k-1}$ factors out of $m_{10}+m_{11}$. Similar argument justifies the coefficient of $t_2^{2^{k-1}}$.	   	
\end{proof}

%\begin{theorem} \label{obstruction}
%	Let $\pi$ be a real representation of $C_n \times C_n,\, n\equiv 0\pmod 4$, of the form \eqref{eqpi} with the condition $w_1(\pi)=w_2(\pi)=0$. Let $k = \mathrm{min}(\upsilon_2(m_{01}+m_{11}), \upsilon_2(m_{10}+m_{11}), \upsilon_2(m_{01} + m_{10}+m_{11})+1).$ Then the obstruction class for $\pi$ is 
%	\begin{align*}
%	w_{2^{k+1}}(\pi)=& \left(\dfrac{m_{10} + m_{11}}{2^{{k}}}\right)\cdot t_1^{2^{k}}+ \left(\dfrac{m_{01} + m_{11}}{2^{{k}}} \right) \cdot t_2^{2^{k}}\\
%	&+\left(\dfrac{m_{01}+m_{10}+m_{11}}{2^{k-1}}\right)\cdot  t_1^{2^{k-1}}t_2^{2^{k-1}},
%	\end{align*}
%	%Suppose $w_{2^i}(\pi) = 0$ for $i=0,1, \ldots ,k-1$ when $k>1$, then  
%	%	\begin{align*}
%	%	w_{2^k}(\pi)=& \binom{m_{10} + m_{11}}{2^{{k-1}}}t_1^{2^{k-1}}+ \binom{m_{01} + m_{11}}{2^{{k-1}}}t_2^{2^{k-1}}\\
%	%	&+\left(\binom{m_{10} + m_{11}}{2^{{k-1}}} + \binom{m_{01} + m_{11}}{2^{{k-1}}}+\binom{m_{10} + m_{01}}{2^{{k-1}}}\right) t_1^{2^{k-2}}t_2^{2^{k-2}},
%	%	\end{align*}
%	where the terms $m_{01}, m_{10}, m_{11}$ are as in Lemma \ref{character.formula}. 	
%\end{theorem} 	

\begin{theorem} \label{obstruction}
	Consider a real achiral representation $\pi$ of $C_n \times C_n,\, n\equiv 0\pmod 4$, of the form \eqref{eqpi} such that $m_1m_2+m_2m_3+m_3m_1\equiv 0\pmod 2$. Let $k = \mathrm{min}(\upsilon_2(m_{01}+m_{11}), \upsilon_2(m_{10}+m_{11}), \upsilon_2(m_{01} + m_{10}+m_{11})+1)$. If $k=0$, then the obstruction class is $w_2(\pi)=t_1+t_2$. For $k\geq 1$, the obstruction class for $\pi$ is 
	\begin{align*}
	w_{2^{k+1}}(\pi)=& \left(\dfrac{m_{10} + m_{11}}{2^{{k}}}\right)\cdot t_1^{2^{k}}+ \left(\dfrac{m_{01} + m_{11}}{2^{{k}}} \right) \cdot t_2^{2^{k}}\\
	&+\left(\dfrac{m_{01}+m_{10}+m_{11}}{2^{k-1}}\right)\cdot  t_1^{2^{k-1}}t_2^{2^{k-1}}.
	\end{align*} 
	
\end{theorem}

\begin{proof}

We write the proof in two steps.\\	
		
\textbf{Step $1$}.	
We begin by proving that if $w_{2^i}(\pi) = 0$ for $i=0,1,\ldots r-1$ then 
 \begin{equation}\label{eq3.1}
\begin{split}
w_{2^{r}}(\pi)&=\binom{m_{10} + m_{11}}{2^{r-1}}t_1^{2^{r-1}}+\binom{m_{01} + m_{11}}{2^{r-1}}t_2^{2^{r-1}}\\
&+\left(\binom{m_{10} + m_{11}}{2^{r-1}} + \binom{m_{01} + m_{11}}{2^{r-1}}+\binom{m_{10} + m_{01}}{2^{r-1}}\right)t_1^{2^{r-2}}t_2^{2^{r-2}}.
\end{split}
\end{equation}

	First we prove that if $w_{2^l}(\pi) = 0$ for $l=0,1,\ldots,r-1$, then the coefficient of $t_1^i t_2^{2^{r-1}-i}$ in $w_{2^r}(\pi)$ is even except when $i\in \{0,2^{r-2},2^{r-1}\}$. Moreover when $i=2^{r-2}$ then it is $$\left(\binom{m_{10} + m_{11}}{2^{{r-1}}} + \binom{m_{01} + m_{11}}{2^{{r-1}}}+\binom{m_{10} + m_{01}}{2^{{r-1}}}\right).$$
	
	We proceed by induction. The base case $w_4(\pi)$ is proved in Lemma \ref{lem1}. Assume $w_{2^i}(\pi) = 0$ for $i=0,1,\ldots,r-1$. Moreover, assume the statement (\eqref{eq3.1}) is true for $j$ where $1\leq j\leq r-1$. In particular for every $j$ for $1 \leq j \leq r-1$  we have
	\begin{equation}\label{eq3}
	\begin{split}
	w_{2^{j}}(\pi)&=\binom{m_{10} + m_{11}}{2^{j-1}}t_1^{2^{j-1}}+\binom{m_{01} + m_{11}}{2^{j-1}}t_2^{2^{j-1}}\\
	&+\left(\binom{m_{10} + m_{11}}{2^{j-1}} + \binom{m_{01} + m_{11}}{2^{j-1}}+\binom{m_{10} + m_{01}}{2^{j-1}}\right)t_1^{2^{j-2}}t_2^{2^{j-2}}.
	\end{split}
	\end{equation}
	By Equation \eqref{eq3}  we have $\binom{m_{10}+m_{11}}{2^i} \equiv 0 \pmod2$, $\binom{m_{01}+m_{11}}{2^i} \equiv 0 \pmod2$ and $\binom{m_{10}+m_{01}}{2^i} \equiv 0 \pmod2$ for $i=0,1,\ldots ,r-2$. Then by Proposition \ref{binomind} we get $2^{r-1}$ divides all the three quantities $m_{10}+m_{11},  m_{01}+m_{11}, m_{10} + m_{01}$.
%	\begin{equation}\label{eqdiv} 
%	\begin{split} 
%	1) 2^{r-1} &\mid m_{10}+m_{11}\\
%	2) 2^{r-1} &\mid m_{01}+m_{11}\\
%	3) 2^{r-1} &\mid m_{10} + m_{01}.
%	\end{split} 
%	\end{equation}
		This indeed means $2^{r-2} \mid m_{10}, m_{01}, m_{11}$.
	%	We aim to calculate $w_{2^k}(\pi)$ provided $w_1(\pi)=\cdots=w_{2^{k-1}}(\pi)=0$.
	By Lemma \ref{lem0} we have 
\begin{equation*}
w(\pi)=(1+t_1)^{m_{10}}\cdot (1+t_2)^{m_{01}}\cdot (1+t_1+t_2)^{m_{11}}.
\end{equation*}
Expanding the right hand side we have
	\begin{equation*}
	\left(\sum_{a=0}^{m_{10}} \binom{m_{10}}{a}t_1^a\right)\cdot \left(\sum_{b=0}^{m_{01}} \binom{m_{01}}{b}t_2^b\right)\cdot \left(\sum_{l=0}^{m_{11}} \binom{m_{11}}{l}(t_1+t_2)^l\right).
	\end{equation*}

%	We have the coefficient of $t_1^{2^{k-2}}t_2^{2^{k-2}}$
%	as
%	
%	\begin{equation*}
%	\sum_{\{a,b \mid 0\leq a \leq 2^{k-2}, 0\leq b \leq 2^{k-2}\}}\left( \binom{m_{10}}{a}\right)\cdot \left( \binom{m_{01}}{b}\right)\cdot \left( \binom{m_{11}}{2^{k-1}-(a+b)}\binom{2^{k-1}-(a+b)}{2^{k-2}-a}\right).
%	\end{equation*}
%	

	The coefficient of $t_1^it_2^{2^{r-1}-i}$ is 
	
	\begin{equation*}
	\sum_{\{a,b \mid 0\leq a \leq i, 0\leq b \leq 2^{r-1}-i\}}\left( \binom{m_{10}}{a}\right)\cdot \left( \binom{m_{01}}{b}\right)\cdot \left( \binom{m_{11}}{2^{r-1}-(a+b)}\binom{2^{r-1}-(a+b)}{i-a}\right).
	\end{equation*}
	We have $2^{r-2} \mid m_{01}, m_{10}, m_{11}$.
	If $i\neq 2^{r-2}$, then either $i$ or $2^{r-1}-i  $ is strictly less than $2^{r-2}$, either $a$ or $b$ or $2^{r-1}- (a+b)$ is strictly less than $2^{r-2}$ and non zero. Thus one of $\binom{m_{10}}{a}$	or $\binom{m_{01}}{b}$ or $\binom{m_{11}}{2^{r-1}-(a+b)}$ is even , hence the whole summation is even when $i\neq 2^{r-2}$.
	When $i=2^{r-2}$ then $(a=0, b=0)$, $(a=2^{r-2}, b=2^{r-2})$,$(a=0,b=2^{r-2})$,$(a=2^{r-2},b=0)$ these terms survive and they give the desired expression. Here we use Vandermonde identity \eqref{van} to obtain 
	\begin{align*}
	{}&\binom{m_{10}}{2^{r-2}}\cdot \binom{m_{01}}{2^{r-2}}+\binom{m_{10}}{2^{r-2}}\cdot \binom{m_{11}}{2^{r-2}}+\binom{m_{01}}{2^{r-2}}\cdot \binom{m_{11}}{2^{r-2}}\\
	&\equiv \binom{m_{10}+m_{01}}{2^{r-1}}+\binom{m_{10}+m_{11}}{2^{r-1}}+\binom{m_{01}+m_{11}}{2^{r-1}}\pmod 2.
	\end{align*}
	We obtain the coefficients for $t_i^{2^{r-1}}$ by Lemma \ref{first2} for $1\leq i\leq 2$.
	
\textbf{Step $2$}. Note that $2^{r-1}$ divides the three quantities $m_{10}+m_{11}, m_{01}+m_{11}, m_{10}+m_{01}$. Using Proposition \ref{binom1} we have 

\begin{equation}\label{eq10}
\begin{split}
w_{2^r}(\pi)=& \left(\dfrac{m_{10} + m_{11}}{2^{{r-1}}}\right)\cdot t_1^{2^{r-1}}+ \left(\dfrac{m_{01} + m_{11}}{2^{{r-1}}} \right) \cdot t_2^{2^{r-1}}\\
&+\left(\dfrac{m_{10} + m_{11}}{2^{{r-1}}} + \dfrac{m_{01} + m_{11}}{2^{{r-1}}}+\dfrac{m_{10} + m_{01}}{2^{{r-1}}}\right)\cdot  t_1^{2^{r-2}}t_2^{2^{r-2}},
\end{split}
\end{equation}
provided $w_i(\pi)=0$, for $1\leq i\leq 2^{r-1}$. If
$$k=\mathrm{min}(\upsilon_2(m_{01}+m_{11}), \upsilon_2(m_{10}+m_{11}), \upsilon_2(m_{01} + m_{10}+m_{11})+1),$$
then for $r<k+1$, the coefficients of $t_1^{2^{r-1}}, t_2^{2^{r-1}}, t_1^{2^{r-2}}t_1^{2^{r-2}}$ in Equation \eqref{eq10} are even.
Therefore $w_{2^{k+1}}(\pi)$ is the obstruction class. Hence we have the  expression for $w_{2^{k+1}}(\pi)$ by putting $r=k+1$ in Equation \eqref{eq10}.
\end{proof}

The analogous result for $n\equiv 2\pmod 4$ is as follows:

\begin{theorem}\label{c2obs}
Consider a real representation $\varphi$ of $C_n \times C_n$ where $n \equiv 2 \pmod 4$, as in \eqref{eq4}. Let $m_{01}, m_{10}, m_{11}$ be as in Lemma \ref{character.formula.C2} and take 
$$k=\mathrm{min}(\upsilon_2(m_{01}+m_{11}), \upsilon_2(m_{10}+m_{11}), \upsilon_2(m_{01} + m_{10}+m_{11})+1).$$
Then the obstruction class for $\pi$ is 
\[
w_{2^k}(\varphi)=
\begin{cases}

v_1+v_2,\quad\text{if $k=0$},\\\\
\left( \dfrac{m_{10} + m_{11}}{2^{{k}}} \right) v_1^{2^{k}}+ \left(\dfrac{m_{01} + m_{11}}{2^{{k}}} \right) v_2^{2^{k}}
	+\left(\dfrac{m_{10}+m_{01}+  m_{11}}{2^{{k-1}}} \right) v_1^{2^{k-1}}v_2^{2^{k-1}},\quad\text{if $k\geq 1$}.
\end{cases}	
\]		
	
%	Suppose $w_{2^i}(\pi) = 0$ for $i=0,1, \ldots ,k-1$ when $k>1$, then  
%	\begin{align*}
%	w_{2^k}(\pi)=& \binom{m_{10} + m_{11}}{2^{{k}}}v_1^{2^{k}}+ \binom{m_{01} + m_{11}}{2^{{k}}}v_2^{2^{k}}\\
%	&+\left(\binom{m_{10} + m_{11}}{2^{{k}}} + \binom{m_{01} + m_{11}}{2^{{k}}}+\binom{m_{10} + m_{01}}{2^{{k}}}\right) v_1^{2^{k-1}}v_2^{2^{k-1}}.
%	\end{align*}  
\end{theorem}

\begin{proof}
%We have 
%\begin{equation}
%w(\varphi) = (1+v_1)^{m_{10}}(1+v_2)^{m_{01}}(1+v_1 +v_2)^{m_{11}}.
%\end{equation}
%Collecting the degree $2$ terms and using Vandermonde identity we have
%\begin{equation*}
%\begin{split}
%w_2(\varphi)&=\binom{m_{10}+m_{11}}{2}v_1^2+\binom{m_{01}+m_{11}}{2}v_2^2\\
%&+\left(\binom{m_{10} + m_{11}}{2}+\binom{m_{01} + m_{11}}{2}+\binom{m_{10} + m_{01}}{2}\right) v_1v_2.
%\end{split}
%\end{equation*}

The proof is similar to that of Theorem \ref{obstruction}. 
\end{proof}

\section{Obstruction class
	of Real Representations of $\mathrm{GL}_2(\mathbb{F}_q)$}\label{5}

Recall that a real representation $\pi$ of a finite group $G$ is achiral  iff $w_1(\pi)=0$ (see Definition \ref{achiral}.
We apply the results for bicyclic groups  in Section \ref{bicyclic} to calculate obstruction classes for achiral representations $\pi$ of $\mathrm{GL}_2(\mathbb{F}_q)$. Note that for a chiral real representation $\pi$ of $\mathrm{GL}_2(\mathbb{F}_q)$ the obstruction class is $w_1(\pi)$.  

\subsection{Catalogue of Irreducible Representations of $\GL_2(\F_q)$}

We follow the notation of \cite[Section $5.1$]{spjoshi} to enumerate the irreducible representations of $G=\GL_2(\mb F_q)$. Let $\sgn_G=\sgn\circ \det_G$, where $\det_G:\GL_2(\F_q)\to \F_q^{\times}$ denotes the determinant map. We have $\mathrm{Hom}(G,\ztwo)\cong H^1(G, \ztwo)$. Therefore $\sgn_G\in H^1(G, \ztwo)$.

%
%Let $\St_G$ be the  Steinberg representation, so that $\Ind_B^G {\bold 1}={\bold 1} \oplus \St_G$. For $\chi \neq \chi'$ linear characters of $\mb F_q^\times$, let $\pi(\chi,\chi')$ be the  parabolic induction $\Ind_B^G (\chi \boxtimes \chi')$. 
%
%If $\theta$ is a character of $\mb F_{q^2}^\times$, write $\theta^\tau$ for its composition with the nontrivial element $\tau$ of the Galois group of $\mb F_{q^2}$ over $\mb F_{q}$.  We say that $\theta$ is \emph{regular}, provided $\theta \neq \theta^\tau$.
%Fix a nontrivial character $\psi$  of $N$, and define a linear character of $ZN$ by $\theta_{\psi}(zn) = \theta(z) \psi(n)$. 
%For $\theta$ regular, put
%\begin{equation}
%\pi_\theta =\Ind_{ZN}^G \theta_\psi - \Ind_{T}^G \theta.
%\end{equation} 
%These are the cuspidal representations.  
%

The irreducible representations of $G$ are as follows:

\begin{enumerate}
	\item The linear characters
	\item The principal series representations $\pi(\chi,\chi')$, with $\chi \neq \chi'$  characters of $\mb F_q^\times$
	\item Twists $ \St_G \otimes \chi$ of the Steinberg for a linear character $\chi$
	\item The cuspidal representations $\pi_\theta$, with $\theta$ a  regular character of anisotropic torus  $T$.
\end{enumerate}

%All self-dual irreducible representations of $G$ are orthogonal.  

The irreducible orthogonal representations of $G$ are:
\begin{enumerate}
	\item ${\mathbb 1}$ and $\sgn_G$
	\item  $\pi({\mathbb 1},\sgn)$
	\item  $\pi(\chi,\chi^{-1})$ with $\chi$ not quadratic.
	\item $\St_G$ and $\St_G \otimes \sgn_G$
	\item  $\pi_\theta$, where $\theta^\tau =\theta^{-1}$, where $\tau\in \mathrm{Gal}(\F_{q^2}/\F_q)$ is the nontrivial element.
\end{enumerate}

%\begin{theorem}
%	Let $\rho$ be an orthogonal representation of $\GL_2(\F_q)$, with $q \equiv 1 \pmod 4$ such that $w_1(\rho) = w_2(\rho) = 0$.
%	Let $m_{01}= m_{10}, m_{11}$ be as defined above. Suppose $k=\min{\{\upsilon_2(m_{01}+m_{11}),\upsilon_2(m_{01})+1\}}$. Then 
%	\begin{enumerate}
%		\item $w_i(\rho) = 0 $ for $1 \leq i \leq 2^{k+1}-1$
%		\item The first obstruction is $w_{2^{k+1}}(\rho) = \left(\dfrac{m_{01} + m_{11}}{2^k}\right)(t_1^{2^k} + t_2^{2^k})+\left(\dfrac{ m_{01}}{2^{k-1}}\right)t_1^{2^{k-1}}t_2^{2^{k-1}}$.
%	\end{enumerate}
%\end{theorem}

\subsection{Proof of the Main Theorem}

Next we give a proof of our main theorem \ref{main}.

\begin{proof}[Proof of Theorem \ref{main}]
%We calculate $\res (w_i(\rho))$.	
We apply Theorem \ref{obstruction} to the representation $\rho\mid_D$.
From Equation  \eqref{eq10} we have	

\begin{equation}\label{eq11}
\begin{split}
w_{2^r}(\rho)=& \left(\dfrac{m_{10} + m_{11}}{2^{{r-1}}}\right)\cdot t_1^{2^{r-1}}+ \left(\dfrac{m_{01} + m_{11}}{2^{{r-1}}} \right) \cdot t_2^{2^{r-1}}\\
&+\left(\dfrac{m_{10} + m_{11}}{2^{{r-1}}} + \dfrac{m_{01} + m_{11}}{2^{{r-1}}}+\dfrac{m_{10} + m_{01}}{2^{{r-1}}}\right)\cdot  t_1^{2^{r-2}}t_2^{2^{r-2}},
\end{split}
\end{equation}
provided $w_i(\rho)=0$, for $1\leq i\leq 2^{r-1}$. Since $m_{10}=m_{01}$ one has
\begin{equation}\label{eq12}
w_{2^r}(\rho)=
%& \left(\dfrac{m_{01} + m_{11}}{2^{{r-1}}}\right)\cdot t_1^{2^{r-1}}+ \left(\dfrac{m_{01} + m_{11}}{2^{{r-1}}} \right) \cdot t_2^{2^{r-1}}\\
%&+\left(2\cdot \dfrac{m_{01} + m_{11}}{2^{{r-1}}}+\dfrac{ 2m_{01}}{2^{{r-1}}}\right)\cdot  t_1^{2^{r-2}}t_2^{2^{r-2}}\\
 \left(\dfrac{m_{01} + m_{11}}{2^{{r-1}}}\right)(t_1^{2^{r-1}}+t_2^{2^{r-1}})
+\left(\dfrac{2m_{01}}{2^{r-1}}\right)\cdot t_1^{2^{r-2}}t_2^{2^{r-2}}.
\end{equation}
If
$$k=\mathrm{min}(\upsilon_2(m_{01}+m_{11}),  \upsilon_2(m_{01})+1),$$
then for $r<k+1$, the coefficients of $t_1^{2^{r-1}}, t_2^{2^{r-1}}, t_1^{2^{r-2}}t_1^{2^{r-2}}$ in Equation \eqref{eq12} are even.
Therefore $w_{2^{k+1}}(\pi)$ is the obstruction class.
%   We prove this by induction. We use the detection result \eqref{detect}. So we just calculate $w_i(\rho\mid_D)$ instead of $w_i(\rho)$.
%   Let $w_1(\rho)=w_2(\rho)=0$, and $k= min(\upsilon_2(m_{01} + m_{11}), \upsilon_2(m_{01})+1)$. Moreover if $k=1$ then from Lemma \ref{lem1} we have
%	\begin{equation}
%	w_{4}(\rho)=\binom{m_{10}+m_{11}}{2}(t_1^2 + t_2^2)	+\binom{2m_{01}}{2}t_1t_2,
%	\end{equation} which is the base case of the induction.
%	If $k\geq 2$, then we have $w_4(\rho)=0$. In particular, if $w_{2^j}(\rho)=0$ for $0\leq j\leq i$, then from Theorem \ref{obstruction} and Proposition \ref{binom1} one gets
%	$$w_{2^{i+1}}(\rho) = \left(\dfrac{m_{01} + m_{11}}{2^i}\right)(t_1^{2^i} + t_2^{2^i})+\left(\dfrac{ m_{01}}{2^{i-1}}\right)t_1^{2^{i-1}}t_2^{2^{i-1}}.$$
%	Moreover, if $i<k$, then $w_{2^{i+1}}(\rho)=0$. Then we can again apply above result to $i+1$ instead of $i$. This process will stop when $i=k$. We obtain the theorem by using induction on $i$ and the results from Theorem \ref{obstruction} and Proposition \ref{binom1}.
%	%the fact that if $2^a \mid b$ then $\dbinom{b}{2^a} \equiv \dfrac{b}{2^a} \pmod 2$, for $a,b\in \Z_{\geq 0}$.
\end{proof}

%\begin{corollary}
%Let $\rho$ be a real representation of $\GL_2(\F_q), q \equiv 1 \pmod 4$  such that $m_{01}=m_{10}=0$ , and $\upsilon_2(m_{11})=k$. If $w_1(\rho)=w_2(\rho)=0$.  then
%	\begin{enumerate}
%		
%		\item 
%		$w_{2^i}(\rho)=0$, for $0\leq i \leq k$.
%		\item The obstruction class is
%		$w_{2^{k+1}}(\rho)=t_1^{2^{k}} + t_2^{2^{k}}$.
%		\item 
%		$w_i(\rho)=\binom{m_{11}}{i}(t_1+t_2)^i$, for $2^{k+1}<i\leq\dim \rho$.
%		%\item 
%		%		$w_{2^{k}}(\rho)=t_1^{2^{k-1}} + t_2^{2^{k-1}}$, if $m_{11}\geq 2^{k-1}$
%		%		and $l-1\equiv 2\pmod 4$, where $m_{11}=2^{k-2}l$. Otherwise $w_{2^{k}}(\rho)=0$.
%		%		
%		
%		%\item 
%		%		$w_j(\rho)=\binom{m_{11}}{j}(t_1+t_2)^j$, for $2^{k}<j\leq \dim \rho$.
%	\end{enumerate}
%\end{corollary}
%\begin{proof}
%The results in $(1)$ and $(2)$ follows directly from Theorem \ref{main}.	For $(3)$ observe that $w(\rho) = (1+t_1+t_2)^{m_{11}}$.
%\end{proof}

%\begin{corollary}
%	Let $\rho$ be an orthogonal representation of $GL(2,q)$, with $q \equiv 1 \pmod 4$ , such that $w_1(\rho)= w_2(\rho)=0$. Moreover if $m_{11}=0$ and $k =\upsilon_2(m_{01})$. Then 
%	\begin{enumerate}
%		\item $w_i(\rho) = 0$ for $1\leq i \leq 2^{k+1}-1$,
%		\item The first obstruction is $w_{2^{k+1}}(\rho) = t_1^{2^k}+t_2^{2^k}$.
%	\end{enumerate}
%\end{corollary}

The next result gives the obstruction classes for real representations of $\GL_2(\F_q)$, where $q \equiv 3 \pmod 4$.

\begin{theorem} \label{main2}
	Let $\rho$ be an orthogonal representation of $\GL_2(\F_q)$, with $q \equiv 3 \pmod 4$.
	Let $m_{01}= m_{10}, m_{11}$ be as defined above. Suppose $$k=\min{\{\upsilon_2(m_{01}+m_{11}),\upsilon_2(m_{01})+1\}}.$$ 
	Then the obstruction class is
\[
w_{2^{k}}(\rho) =
\begin{cases}
e_1,\quad\text{if $k=0$},\\\\
c_1,\quad\text{if $k=1$},\\\\
\left(\dfrac{m_{01} + m_{11}}{2^k}\right)e_1^{2^{k}} +\left(\dfrac{ m_{01}}{2^{k-1}}\right)c_2^{2^{k-2}},\quad\text{if $k\geq 2$}.
\end{cases}
\]	
%	\begin{enumerate}
%		\item If $k=0$, then $w_1(\rho)=v_1+v_2$.
%		\item If $k>0$, $w_i(\rho) = 0 $ for $1 \leq i \leq 2^{k}-1$.
%		\item If $k>0$, the obstruction class is $w_{2^{k}}(\rho) = \left(\dfrac{m_{01} + m_{11}}{2^k}\right)(v_1^{2^{k}} + v_2^{2^{k}})+\left(\dfrac{ m_{01}}{2^{k-1}}\right)v_1^{2^{k-1}}v_2^{2^{k-1}}$.
%	\end{enumerate}
\end{theorem}

\begin{proof}
	The proof is analogous to that of Theorem \ref{main}.
\end{proof}

We aim to calculate the first and second Stiefel Whitney classes of orthogonally irreducible representations of $G$.
%Consider the repersentation $\pi(1,\sgn)$ of $\mathrm{GL}_2(\mathbb{F}_q)$ for $q\equiv 1\pmod 8$. 
The paper \cite{spjoshi} gives criteria to determine spinorial representations  of $G$. We know that an orthogonal representation $\pi$ of $G$ is spinorial iff $w_2(\pi)=w_1(\pi)\cup w_1(\pi)$ (see Proposition \ref{spin.swc}). Moreover $H^2(GL_2(\F_q), \ztwo) = \ztwo$ (see \cite{spjoshi}). Hence if we know $w_1$ we also can derive $w_2$.

\subsection{Calculation of $w_1$}

%We will need to compute $\det \pi$ when $\omega_\pi=\sgn_Z$. 
According to \cite[Proposition 29.2] {bushnell}, if $H$ is a subgroup of a finite group $G$, and $\rho$ is a representation of $H$, then we have
\begin{equation} \label{Bush}
\det(\Ind_H^G(\rho)) =\det(\Ind_H^G {\mathbb{ 1}})^{\dim \rho}\cdot (\det \rho \circ \ver_{G/H}),
\end{equation}
where
\begin{equation}
\ver_{G/H} : G/D(G) \to H/D(H)
\end{equation}
is the usual ``verlagerung" or transfer map, with $D(G)$ the derived subgroup of $G$.

We recall that given $g \in G$ we have 

\begin{equation} \label{ver3}
\ver_{G/H}(g \cdot D(G)) = \prod_{x\in G/H} h_{x,g} \cdot D(H).
\end{equation}

Here $h_{x,g}$ is defined as follows: Pick a section $t: G/H \to G$  of the canonical projection.
Given $g \in G$ and   $x \in G/H$ there is a $y \in G/H$ and an $h_{x,g} \in H$ such that $gt(x) = t(y)h_{x,g}$.

Let $B$ denote the Borel subgroup of $GL_2(\F_q)$.

\begin{lemma} \label{ver1} The transfer map corresponding to the subgroup $B<G$ is given by
	\begin{equation}
	\ver_{G/B}(g  \: \Mod D(G))=\mat{\det g}00{\det g}  \: \Mod D(B).
	\end{equation}
\end{lemma}

\begin{proof} 
%	Using the standard identification of the projective line $\mb P^1=\mb P^1(\mb F_q)=\mb F_q \cup \{ \infty \} \cong G/B$, we may replace the right hand side of \eqref{ver3}
%	with a product over $x \in \mb P^1(\mb F_q)$.
	
	It is well-known that $G/B = \underset{x \in \F_q}{\bigcup}\mat 10x1 B\bigcup \mat 0110B$. 
	 
	Define $t': G/B \to G$ by
	\begin{equation}
	t'(\alpha)= \begin{cases} 
	\mat 10x1 \text{ if $\alpha \in \mat 10x1B$}, \\ \\
	\mat 0110 \text{ if $\alpha\in\mat 0110B$}. \\
	\end{cases}
	\end{equation}
	
	Since $D(G)= \SL_2(\F_q)$, we may take $g=\mat a001$ with $a \in \F_q^\times$.
	Let $\alpha = \mat 10x1B$. Then $g t'(\alpha)=t'(\beta) h_{\alpha,g}$ where $\beta = \mat 10{x/a}1B$ and $h_{\alpha,g}=\mat a001$.
	For $\alpha=\mat 0110$ we have $gt'(\alpha)=t'(\alpha)h_{\alpha,g}$ with $h_{\alpha,g}=\mat 100a$.
	Thus $\ver_{G/B}\left( \mat a001 \right)= \mat {a^q}00a = \mat a00a$, whence the proposition.
\end{proof}

\begin{lemma} \label{some.dets} We have
	\begin{enumerate}
		\item $\det \Ind_B^G {\bold 1}=\sgn_G$
		\item $\det \St_G=\sgn_G$
		\item $\det \pi({\bold 1},\sgn)={\bold 1}$.
	\end{enumerate}
\end{lemma}

\begin{proof} 
	Note that $\Ind_B^G {\mathbb{1}}$ equals $\C[G/B]$, the permutation representation of $G$ corresponding to its action on $G/B$. Let $a$ be a generator of $\mb F_q^\times$ and $g_a=\mat a001$. Then $g_a$ fixes $\mat 1001B, \mat 0110B\in G/B$ and acts like a $(q-1)$ cycle on the complement. It follows that $\det(g_a; \C[G/B])=-1$ and therefore $\det \Ind_B^G {\mathbb{1}}=\sgn_G$ as claimed. This entails that $\det \St_G=\sgn_G$ as well.
	Finally by \eqref{Bush}, we have
	\begin{equation}
	\begin{split}
	\det(\Ind_B^G({\mathbb{ 1}} \boxtimes \sgn)) &=\det(\Ind_B^G {\mathbb{1}}) \cdot ( ({\mathbb{1}} \boxtimes \sgn) \circ \ver_{G/H}) \\
	&=\sgn_G \cdot \sgn_G\\
	&= {\mathbb{1}}.
	\end{split}
	\end{equation}
\end{proof}

%Let $H$ be a subgroup of a finite group $G$. Then one can define the Verlegurg map $\mathrm{ver}_{G/H}:G/D(G)\to H/D(H)$ as...
%Let$\rho$ be a representation of $H$. Then 
%\begin{equation}\label{ver}
%\det(\mathrm{Ind}_H^G\rho)=\det(\mathrm{Ind}_H^G\mathbb{1})^{\dim\rho}\cdot (\det\rho\circ \mathrm{ver}_{G/H})
%\end{equation}
%
%
%\begin{lemma}\label{ver1}
%	We have
%	$$ \mathrm{ver}_{G/B}(g\cdot D(G))=\begin{pmatrix}
%	\det g & 0\\
%	0 & \det g
%	\end{pmatrix}\cdot D(B)$$
%\end{lemma}

\begin{lemma}
	We have $$w_1(\Ind_B^G(\chi \otimes \chi^{-1}))=\sgn_G.$$
\end{lemma}

\begin{proof}
Now consider the representation $\chi\otimes\chi^{-1}$ of $B$. We have
$$\chi\otimes\chi^{-1}\begin{pmatrix}
a_{11} & a_{12}\\
0 & a_{22}
\end{pmatrix}=\chi(a_{11})\chi^{-1}(a_{22}).$$
From Equation \eqref{Bush} we obtain
$$\det(\mathrm{Ind}_B^G(\chi\otimes\chi^{-1}))=\det(\mathrm{Ind}_B^G\mathbb{1})^{\dim\chi\otimes\chi^{-1}}\cdot (\det(\chi\otimes\chi^{-1})\circ \mathrm{ver}_{G/B}).$$
We have $\det(\mathrm{Ind}_B^G\mathbb{1})=\sgn_G$. Using Lemma \ref{ver1} we compute 
\begin{align*}
\det(\chi\otimes\chi^{-1})\circ \mathrm{ver}_{G/B}&=\chi(\det g)\cdot\chi^{-1}(\det g)\\
&=1.
\end{align*}
Therefore $w_1(\mathrm{Ind}_B^G(\chi\otimes\chi^{-1}))=\sgn_G$.
\end{proof}

Let $\pi_{\theta}$ denote the cuspidal representation of $G=\mathrm{GL}(2,q)$, where $\theta:\mathbb{F}_{q^2}^{\times}\to \C^{\times}$ is as defined in \cite[Section $6.4$]{bushnell}.
% We write $D$ to denote the diagonal torus and $Z$ for the center of $G$. Since the map $\res^*:H^*(G,\ztwo)\to H^*(D,\ztwo)$ is injective, it is enough to calculate $\det(\mathrm{Res}\mid_D \pi_{\theta})$.
We denote the center of $G$ by $Z$.
\begin{prop}
	We have
	$$\mathrm{Res}\mid_D \pi_{\theta}=\mathrm{Ind}_Z^D \theta.$$	
\end{prop}
\begin{proof}
	According to \cite[page 47]{bushnell} we have
	
	$$\pi_{\theta} = \Ind^G_{ZN}(\psi_{\theta}) - \Ind^G_T \theta,$$ where $\psi_{\theta}(zn) = \theta(z) \psi(n).$ Here $\psi$ is a character of $N$. Observe that $gZNg^{-1} \cap D = Z = D \cap gTg^{-1}$ using eigen values argument for both. We have
	
	\begin{align*}
	Res_D \pi(\theta) &= Res_D \Ind^G_{ZN} \psi_{\theta} - Res_D \Ind^G_T \theta\\
	&=\bigoplus_{g \in D\backslash G / ZN} \Ind^D_{D \cap gZNg^{-1}} \psi_{\theta}^{g}- \bigoplus_{g \in D\backslash G / T} \Ind^D_{D \cap gTg^{-1}} \theta^{g}\\
	&= \bigoplus_{g \in D\backslash G / ZN} \Ind^D_{Z} \psi_{\theta}- \bigoplus_{g \in D\backslash G / T} \Ind^D_Z \theta\\
	&= (|D \backslash G / ZN|-|D\backslash G/ T|) \Ind^D_Z \theta\\
	&=(\dfrac{|G|\cdot|Z|}{|D|\cdot|ZN|}- \dfrac{|G|\cdot|Z|}{|D|\cdot|T|})\Ind^D_Z \theta\\
	&=\Ind^D_Z \theta .
	\end{align*}
	
\end{proof}

\begin{lemma}
	We have $w_1(\pi_{\theta})= \sgn_G$.
\end{lemma}
\begin{proof}
From Equation \eqref{Bush} we have
\begin{equation}
\det(\mathrm{Ind}_Z^D\theta)=\det(\mathrm{Ind}_Z^D\mathbb{1})^{\dim\theta}\cdot (\det\theta\circ \mathrm{ver}_{D/Z})
\end{equation}
Using Frobenius reciprocity theorem we deduce that $$\mathrm{Ind}_Z^D\mathbb{1}=\bigoplus\limits_{\chi\in\mathbb{F}_q^{\times}}\chi\otimes \chi^{-1}.$$
This gives $\det(\mathrm{Ind}_Z^D\mathbb{1})=\sgn_D$.
Consider the map $t\left(\begin{pmatrix}
a & 0\\
0 & b
\end{pmatrix} Z\right)=\begin{pmatrix}
ab^{-1} & 0\\
0 & 1
\end{pmatrix}$.
This gives $h_{xg}=\begin{pmatrix}
g_2 & 0\\
0 & g_2
\end{pmatrix}$, where $g=\begin{pmatrix}
g_1 & 0\\
0 & g_2
\end{pmatrix}$.
Therefore
\begin{align*}
\mathrm{ver}_{D/Z}(g)&=\prod_{x\in D/Z}h_{xg}\\
&=g_2^{q-1}I_2\\
&=I_2.
\end{align*}

Therefore $w_1(\pi_{\theta})=\sgn_G$.
\end{proof}

A complete list of spinorial (see Section \ref{oir}) OIRs of $\GL_2(\F_q)$ is given in \cite[Theorem $4$]{spjoshi}. 
Let $b_1$ denote the nonzero element in $H^1(\GL_2(\F_q),\ztwo) \cong  \Z/2\Z$ and $b_2$ be the nonzero element in $H^2(\GL_2(\F_q)) \cong \Z/2\Z$ (see \cite[Proposition $7$]{spjoshi}).  Note that $b_1 = s_1 + s_2$ and $b_2=t_1+t_2$ in $H^*(D)$. Therefore $b_1^2 =0$. Hence $w_2(\pi) = 0$ iff $\pi$ is spinorial. 
Using \cite[Theorem $4$]{spjoshi} we have the following example. 
For $q \equiv 1 \pmod 4$ the representation $\pi(\chi,\chi^{-1})$ is spinorial iff 
\begin{enumerate}
	\item $\chi$ is odd and $q \equiv 5 \pmod 8$ or
	\item  $\chi$ is even and $q \equiv 1 \pmod 8$. 
\end{enumerate}
This
clearly gives the formula $w_2(\pi) = \left((q-1)/4 + \epsilon_{\chi}\right)b_2$. Similarly we can calculate $w_2$ for other representations.

\begin{center}
	\begin{table}[ht]
	\caption{ First and Second Stiefel Whitney Classes for OIRs of $\mathrm{GL}_2(\F_q)$}
	\centering	
	
	\begin{tabular}{|c|c|r|}
		\hline
		$\pi$ & $w_1(\pi)$ & $w_2(\pi)$ \\
		\hline
		${\bf 1}$ & $0$ & $0$ \\
		\hline
		$\sgn_G$ & $b_1$ & $0$ \\
		\hline
		$\pi(\chi,\chi^{-1})$ & $b_1$ & $\left(\frac{q-1}{4} + \epsilon_\chi \right)b_2$ if $q \equiv 1 \mod 4$ \\
		& & $(\half(\frac{q-1}{2}-1)+ \epsilon_\chi) b_2$ if $q \equiv 3 \mod 4$ \\
		\hline
		$\pi(1,\sgn)$,  $\St_G \otimes \sgn_G$ & $0$ & $\frac{q-1}{4} b_2$ if $q \equiv 1 \mod 4$ \\
		&& $\half \left( \frac{q-1}{2}+ 1 \right) b_2$ if $q \equiv 3 \mod 4$ \\
		\hline
		$\pi_\theta$, $\St_G$ & $b_1$ &  $\frac{q-1}{4} b_2$ if $q \equiv 1 \mod 4$ \\
		&& $\half \left( \frac{q-1}{2}- 1 \right) b_2$ if $q \equiv 3 \mod 4$ \\
		\hline
		$S(\chi \circ \det_G), \: \chi^2 \neq 1$ & $0$ & $\epsilon_\chi \cdot b_2$ \\
		\hline
		$S(\St_G \otimes \chi)$ &  $0$ &$(\epsilon_\chi + \half(q-1))b_2$ \\
		\hline
		$S(\pi(\chi_1,\chi_2))$ & $0$ & $(\epsilon_{\chi_1 \cdot \chi_2}  + \half(q-1)) b_2$\\
		\hline
		$S(\pi_\theta)$ &  $0$ &$\frac{1}{2}(q-1) \cdot b_2$ \\
		\hline
	\end{tabular}
	\label{table1}
\end{table}
\end{center}

\newpage

\begin{center}
	\begin{table}[ht]
		\caption{ Table showing the Obstruction class of the achiral, spinorial OIRs of $\GL_2(\F_q)$}
		\centering	
		
		\tabcolsep=.6cm
		\renewcommand{\arraystretch}{2}
		\begin{tabular}{|p{2cm}| p{5cm}|p{5cm}| }
			\hline
		$\rho$	& $k$ & Obstruction Class for $k\geq 2$ \\
		\hline
			$\pi(1,\sgn)$ &$\upsilon_2(q-1)-2$, if $q\equiv 1\pmod 8$  \newline\newline $\upsilon_2(q+1)-1$, if $q\equiv 7\pmod 8$  & $c_1^{2^k}$, if $q\equiv 1\pmod 8$ \newline\newline $e_1^{2^{k}}$, if $q\equiv 7\pmod 8$ \\
			\hline
			$\St_G \otimes \sgn_G$ &$\upsilon_2(q-1)-2$, if $q\equiv 1\pmod 8$ \newline \newline $\upsilon_2(q+1)-1$, if $q\equiv 7\pmod 8$ &$c_1^{2^k}$, if $q\equiv 1\pmod 8$ \newline\newline $e_1^{2^{k}}$, if $q\equiv 7\pmod 8$ \\
			\hline
			$S(\chi), \: \chi^2 \neq 1$ &  -&- \\
			\hline
			$S(\St_G \otimes \chi)$ & $\upsilon_2(q-1)-1$, if $q\equiv 1\pmod 4$ \newline\newline $\upsilon_2(q+1)$, if $q\equiv 3\pmod 4$ & $c_1^{2^k}$, if $q\equiv 1\pmod 4$ \newline \newline $e_1^{2^{k}}$, if $q\equiv 3\pmod 4$ \\
			\hline
			$S(\pi(\chi_1,\chi_2))$ & $\upsilon_2(q-1)-1$, if $q\equiv 1\pmod 4$, $\epsilon_{\chi_1}=\epsilon_{\chi_2}=0$  \newline\newline $ \upsilon_2(q+3)-1$, if $q\equiv 1\pmod 4$, $\epsilon_{\chi_1}=\epsilon_{\chi_2}=1$ \newline\newline $\upsilon_2(q+1)+1$ if $q\equiv 3\pmod 4$& $c_1^{2^k}$, if $q\equiv 1\pmod 4$,  $\epsilon_{\chi_1}=\epsilon_{\chi_2}=0$  \newline\newline $c_1^{2^k}$, if $q\equiv 1\pmod 4$, $\epsilon_{\chi_1}=\epsilon_{\chi_2}=1$ \newline\newline $e_1^{2^{k}}+ c_2^{2^{k-2}}$ if $q\equiv 3\pmod 4$
			 \\
			\hline
			$S(\pi_{\theta})$  & $\upsilon_2(q-1)-1$, if $q\equiv 1\pmod 4$    & $c_1^{2^k}$, if $q\equiv 1\pmod 4$, $\theta(-1)=1$ \newline\newline
			$c_1^{2^k}+ c_2^{2^{k-2}}$, if $q\equiv 1\pmod 4$, $\theta(-1)=-1$ \\
			\hline
		\end{tabular}
		\label{table2}
	\end{table}
\end{center}

\section{ Some Applications}\label{app}

\subsection{Steenrod Squares}
For a finite group $G$ one can define Steenrod Square operator $\Sq^i:H^j(G,\Z/2\Z)\to H^{i+j}(G,\Z/2\Z)$ (see \cite[page $52$]{adem} for details). For $x\in H^j(G,\ztwo)$ one has
$$\Sq(x)=\sum_{i=0}^{\infty}\Sq^i(x).$$
Note that $\Sq^i(x)=0$ if $i>j$. Here we mention some basic results of the Steenrod square operator on the cohomolohy ring $H^*(C_n,\ztwo)$ where $n$ is even.
We follow \href{https://bchetard.files.wordpress.com/2017/10/group_cohomology.pdf}{the note} by B. I. Chetard closely.

We know $H^*(C_2, \Z/2\Z)=\Z/2\Z[v]$, where $\deg(v)=1$. 
%We apply the $\Sq$ operator to $v$ to get 
%$$\Sq(v)=v+v^2=v(1+v).$$
We have 
\begin{equation}
\Sq(v^k)=\sum_{i=0}^{k}\binom{k}{i}v^{k+i}.
\end{equation}
We know $H^*(C_{4n},\ztwo)\cong\dfrac{\mathbb{F}_2[s,t]}{(s^2)}$, where $s$ is of degree $1$ and $t$ is of degree $2$. 

Note that $\Sq^1$ is the Bockstein homomorphism associated to the exact sequence 
$$0\longrightarrow \Z/2\Z\longrightarrow \Z/4\Z\longrightarrow \Z/2\Z\longrightarrow 0.$$
This gives the following long exact sequence
\begin{center}
	\begin{tikzpicture}
	\node (A1) at (0,0) {$H^2(C_{4n},\Z/2\Z)$};
	\node (A2) at (3,0) {$H^2(C_{4n},\Z/4\Z)$};
	\node (A3) at (6,0) {$H^2(C_{4n},\Z/2\Z)$};
	\node (A4) at (9,0) {$H^3(C_{4n},\Z/2\Z)$};
	\node (A5) at (11,0) {$\cdots$};
	
	%\path[->,font=\scriptsize,>=angle 90]
	\draw[->](A1)to node [above]{} (A2);
	\draw[->](A2) to node [above]{$\alpha$} (A3);
	\draw[->](A3) to node [above]{$\beta$} (A4);
	\draw[->](A4)to node [above]{} (A5);
	\end{tikzpicture}
\end{center}
Here $\beta=\Sq^1$. We have $\Sq^1(t)=0$ if $t=\alpha(x)$ for some $x\in H^2(C_{4n},\Z/2\Z)$. Since $\Sq^1(s)=0$ the image of the map $\Sq^1:H^1(C_{4n},\Z/2\Z)\to H^2(C_{4n},\Z/2\Z)$ is zero. Moreover we have $H^2(C_{4n},\Z/4\Z)=\Z/4\Z$ and $H^2(C_{4n},\Z/2\Z)=\Z/2\Z$.  Therefore the exact sequence  becomes

\begin{center}
	\begin{tikzpicture}
	\node (A1) at (0,0) {$0$};
	\node (A2) at (2,0) {$\Z/2\Z$};
	\node (A3) at (4,0) {$\Z/4\Z$};
	\node (A4) at (6,0) {$\Z/2\Z$};
	\node (A5) at (8,0) {$\Z/2\Z$.};
	
	%\path[->,font=\scriptsize,>=angle 90]
	\draw[->](A1)to node [above]{} (A2);
	\draw[->](A2) to node [above]{$$} (A3);
	\draw[->](A3) to node [above]{$\alpha$} (A4);
	\draw[->](A4)to node [above]{$\beta$} (A5);
	\end{tikzpicture}
\end{center}
The map $\alpha$ is surjective. Thus $\ker(\beta)=\Z/2\Z$ and this gives $\Sq^1(t)=0$.
 We obtain
\begin{equation}
\Sq(s)=s,\quad\quad \Sq(t^k)=\sum_{i=0}^k\binom{k}{i}t^{k+i}.
\end{equation}

%The proof of the fact $\Sq^1(t)=0$ follows from Section $3.3$ in \href{https://bchetard.files.wordpress.com/2017/10/group_cohomology.pdf}{the note} by B.I. Chetard. 

\subsection{Beyond Obstruction class}

%Let $n \equiv 2 \pmod 4$ and $C_n$ be  the additive  cylic group of order $n$.

%\begin{theorem}\label{c2obs}
%	Consider a real representation $\varphi$ of $C_n \times C_n$ where $n \equiv 2 \pmod 4$. Let $m_{01}, m_{10}, m_{11}$ be as before and take 
%	$$k=\mathrm{min}(\upsilon_2(m_{01}+m_{11}), \upsilon_2(m_{10}+m_{11}), \upsilon_2(m_{01} + m_{10}+m_{11})+1).$$
%	Then the obstruction class for $\pi$ is 
	%\[
	%w_{2^k}(\varphi)=
	%\begin{cases}
	
	%v_1+v_2,\quad\text{if $k=0$},\\\\
	%\left( \dfrac{m_{10} + m_{11}}{2^{{k}}} \right) v_1^{2^{k}}+ \left(\dfrac{m_{01} + m_{11}}{2^{{k}}} \right) v_2^{2^{k}}
	%+\left(\dfrac{m_{10}+m_{01}+  m_{11}}{2^{{k-1}}} \right) v_1^{2^{k-1}}v_2^{2^{k-1}},\quad\text{if $k\geq 1$}.
	%\end{cases}	
	%\]		
	
	%	Suppose $w_{2^i}(\pi) = 0$ for $i=0,1, \ldots ,k-1$ when $k>1$, then  
	%	\begin{align*}
	%	w_{2^k}(\pi)=& \binom{m_{10} + m_{11}}{2^{{k}}}v_1^{2^{k}}+ \binom{m_{01} + m_{11}}{2^{{k}}}v_2^{2^{k}}\\
	%	&+\left(\binom{m_{10} + m_{11}}{2^{{k}}} + \binom{m_{01} + m_{11}}{2^{{k}}}+\binom{m_{10} + m_{01}}{2^{{k}}}\right) v_1^{2^{k-1}}v_2^{2^{k-1}}.
	%	\end{align*}  
%\end{theorem}

%Let $\pi$ be a real representation of $C_n \times C_n$.

\begin{lemma}\label{lem12}
Consider a real representation of $C_n\times C_n$ where $n$ is even. Let $2^l$ be the obstruction degree for $\pi$. Then 
\begin{equation*}
w_{2^l + i}(\pi) = Sq^i(w_{2^l}(\pi)),
\end{equation*} 
for $1 \leq i <2^l$.
\end{lemma}
\begin{proof}
By Wu's formula we have
	
$$Sq^i(w_j)=\sum_{t=0}^i \binom{j+t-i-1}{t}w_{i-t}w_{j+t} .$$
We put $j=2^l$ and use the fact that $w_{i-t}= 0$ unless $i=t$. For $i=t$ we obtain 
$$Sq^i(w_{2^l}) = \binom{2^l-1}{i}w_{2^l+i}= w_{2^k+i}. $$ 
\end{proof}

\begin{theorem}\label{app2}
Consider a real representation $\pi$ of $\GL_2(\F_q)$ where $q\equiv 1\pmod 4$. Let the obstruction degree of $\pi$ is $2^{k+1}$ where $k=\min\{\upsilon_2(m_{10}+m_{11}), \upsilon_2(m_{10})+1\}$. Then we have
\[
w_{2^{k+1}+i}(\pi)=
\begin{cases}
0, \quad \text{if}\,\, 1\leq i<2^{k+1}, i\neq 2^k,\\ \\
\left(\dfrac{m_{10}}{2^{k-1}}\right)(t_1^{2^k}t_2^{2^{k-1}}+t_1^{2^{k-1}}t_2^{2^k}), \quad \text{if}\,\, i=2^k.
\end{cases}
\]
In particular $w_r(\pi)=0$ for $r>\deg\pi$.
\end{theorem}

\begin{proof}
For $q\equiv 1\pmod 4$, the obstruction class of a representation $\pi$ of $\GL_2(\F_q)$ is given by 
$$w_{2^{k+1}}(\pi)=\frac{m_{10}+m_{11}}{2^k}(t_1^{2^k}+t_2^{2^k})+\frac{m_{10}}{2^{k-1}}t_1^{2^{k-1}}t_2^{2^{k-1}},$$
where $k=\min\{\upsilon_2(m_{10}+m_{11}), \upsilon_2(m_{10})+1\}$.  From Lemma \ref{lem12} we know that $w_{2^{k+1} + i}(\pi) = \Sq^i(w_{2^{k+1}}(\pi))$.
Note that the expression for obstruction class is actually $\res(w_{2^{k+1}}(\pi))$,
where $\res$ denotes the restriction of the representation to the diagonal subgroup. Since $\Sq\res=\res\Sq$, we directly compute $\Sq^i(w_{2^{k+1}}(\pi))$ for $i<2^{k+1}$. For $j\in\{1,2\}$ we have

\[
\Sq^i(t_j^{2^k})=
\begin{cases}
0,\quad \text{if $i$ is odd},\\
\binom{2^k}{i/2}t_j^{2^k+i/2},\quad \text{if $i$ is even}.
\end{cases}
\]
Observe that $\binom{2^k}{i/2}=0$, unless $i=0$ or $2^{k+1}$. Since $0<i<2^{k+1}$ we obtain $\Sq^i(t_j^{2^k})=0$.
Using Cartan's formula we have 
\begin{align*}
\Sq^i(t_1^{2^{k-1}}t_2^{2^{k-1}})&=\sum_{l_1+l_2=i}\Sq^{l_1}(t_1^{2^{k-1}})\Sq^{l_2}(t_2^{2^{k-1}})\\
&=\sum_{l_1+l_2=i, l_1,\l_2\, \text{even}}\binom{2^{k-1}}{l_1/2}t_1^{2^{k-1}+l_1/2}\binom{2^{k-1}}{l_2/2}t_1^{2^{k-1}+l_2/2}.
\end{align*}
Note that the expression in the right hand side survives if $l_1=2^k, l_2=0$ or $l_1=0, l_2=2^k$.  This gives 
$$w_{3\cdot 2^k}(\pi)=\left(\dfrac{m_{10}}{2^{k-1}}\right)(t_1^{2^k}t_2^{2^{k-1}}+t_1^{2^{k-1}}t_2^{2^k}).$$

\end{proof}

\begin{theorem}\label{app3}
Consider a real representation of $\GL_2(\F_q)$, where $q\equiv 3\pmod 4$. Let the obstruction degree of $\pi$ is $2^k$ where $k=\min\{\upsilon_2(m_{10}+m_{11}), \upsilon_2(m_{10})+1\}$. Then 
\[
w_{2^k+i}(\pi)=
\begin{cases}
0, \quad \text{if}\,\, 1\leq i<2^{k}, i\neq 2^{k-1},\\ \\
\left(\dfrac{m_{10}}{2^{k-1}}\right)(v_1^{2^k}v_2^{2^{k-1}}+v_1^{2^{k-1}}v_2^{2^k}), \quad \text{if}\,\, i=2^{k-1}.
\end{cases}
\]
In particular $w_r(\pi)=0$ for $r>\deg\pi$.
	%In the setup of the Lemma $\bigstar$ we have $w_{2^k +i} = 0$ for $1 \leq i <2^k, i\neq 2^{k-1}$. For $i = 2^{k-1}$ we have 
	
	%$$w_{2^k + 2^{k-1}}= \left(\dfrac{m_{10}+m_{01}+  m_{11}}{2^{{k-1}}} \right) (v_1^{2^k}v_2^{2^{k-1}} + v_1^{2^{k-1}} v_2^{2^k}). $$
\end{theorem}

\begin{proof}
The proof is similar to that of Theorem \ref{app2}.
	%We have  
%	\begin{align*}
%	Sq^i(w_{2^k}) =& \left( \dfrac{m_{10} + m_{11}}{2^{{k}}} \right) Sq^i(v_1^{2^{k}})+ \left(\dfrac{m_{01} + m_{11}}{2^{{k}}} \right) Sq^i(v_2^{2^{k}})\\
%	&+\left(\dfrac{m_{10}+m_{01}+  m_{11}}{2^{{k-1}}} \right) Sq^i(v_1^{2^{k-1}}v_2^{2^{k-1}}),  
%	\end{align*}
%	We have 
%	\begin{enumerate}
%		\item $Sq^i(v_1^{2^{k}}) = \binom{2^k}{i}v_1^{2^k + i}=0$,
%		\item $Sq^i(v_2^{2^{k}})=\binom{2^k}{i}v_2^{2^k + i}=0 $,
%		\item $ Sq^i(v_1^{2^{k-1}}v_2^{2^{k-1}})= \sum_{l_1+l_2= i} \binom{2^{k-1}}{l_1} \binom{2^{k-1}}{l_2}v_1^{2^{k-1}+l_1}v_2^{2^{k-1}+l_2}.$
%	\end{enumerate}
%	The expression in (3) contributes  non zero terms iff 
%	$i=2^{k-1}$ and $(l_1,l_2)= (0,2^{k-1})$ or $(2^{k-1},0)$.	
	\end{proof}

\subsection{Anisotropic Torus}

Let $a \in \F_{q^2}^{\times}$, $G=\GL_2(\F_q)$. If we identify $\F_{q^2}$ with $\F_q \oplus \F_q$ then the map $x \to a\cdot x$ gives a linear map $M_a \in G  $. Then $M=\{M_a, a \in \F_{q^2}^{\times} \}$ is a cyclic group of order $q^2-1$ and is called an anisotropic torus. 
\begin{theorem}\label{app1}
	  The anisotropic torus $M$ does not detect the mod-2 cohomology of $G$. 	
 \end{theorem} 
\begin{proof}
	Take representation $\pi = S(\chi)$, where $\epsilon_{\chi}=1$. Then $w_2(\pi) \neq 0$.
	 Let $\hat g$ be the generator of $M$. Now we have
	
	$$w_2(\pi\mid_M)= \kappa(c_1(\chi \circ \det\mid_M)).$$
	Since the map $\chi \circ \det\mid_M$ factors through $\F_q^{\times}$ we obtain $\chi \circ \det (\hat{g})=\zeta_{q-1}^j=\zeta_{q^2-1}^{(q+1)j}$. Let $u=c_1(\dot{\chi})$ where $\dot{\chi} \in \mathrm{Hom}(M, \C^{\times})$ such that $\dot{\chi}(\hat{g}) = \zeta_{q^2-1}$. As a result we get 
	\begin{align*}
	\kappa(c_1(\chi \circ \det\mid_M))&= \kappa((q+1)ju)\\
	&=(q+1)j \kappa(u)\\
	&=0.
	\end{align*}
	Hence the proof.
	
	\end{proof}

\bibliographystyle{alpha}
\bibliography{mybib} 

\end{document}